\documentclass[10pt,reqno]{amsart}
\usepackage{amssymb,mathrsfs,graphicx}
\usepackage{ifthen}
\usepackage{hyperref}
\usepackage{scalerel,stackengine}
\usepackage[margin=1in]{geometry}
\usepackage{caption}
\usepackage{sidecap}
\usepackage{rotating,dsfont}
\usepackage{amsthm}
\usepackage{amsfonts}
 
\usepackage{colortbl}
\definecolor{black}{rgb}{0.0, 0.0, 0.0}
\definecolor{red}{rgb}{1.0, 0.5, 0.5}

\numberwithin{equation}{section}

\title[A revisit to the pressureless Euler--Navier--Stokes system]{A revisit to the pressureless Euler--Navier--Stokes system in the whole space and its optimal temporal decay}

\author[Choi]{Young-Pil Choi}
\address[Young-Pil Choi]{\newline Department of Mathematics\newline
Yonsei University, 50 Yonsei-Ro, Seodaemun-Gu, Seoul 03722, Republic of Korea}
\email{ypchoi@yonsei.ac.kr}

\author[Jung]{Jinwook Jung}
\address[Jinwook Jung]{\newline Department of Mathematics and Institute of Pure and Applied Mathematics \newline Jeonbuk National University, 567 Baekje-daero, Deokjin-gu, Jeonju-si, Jeollabuk-do 54896, Republic of Korea}
\email{2jwook12@gmail.com}

\author[Kim]{Junha Kim}
\address[Junha Kim]{\newline School of Mathematics
\newline Korea Institute for Advanced Study, Seoul 02455, Republic of Korea}
\email{junha02@kias.re.kr}

\numberwithin{equation}{section}

\newtheorem{theorem}{Theorem}[section]
\newtheorem{lemma}{Lemma}[section]
\newtheorem{corollary}{Corollary}[section]
\newtheorem{proposition}{Proposition}[section]
\newtheorem{remark}{Remark}[section]

\newcommand{\R}{\mathbb R}

\newcommand{\N}{\mathbb N}
\newcommand{\ls}{\lesssim}

\newcommand{\mc}{\mathcal C}

\newcommand{\bq}{\begin{equation}}
\newcommand{\eq}{\end{equation}}
\newcommand{\e}{\varepsilon}
\newcommand{\lt}{\left}
\newcommand{\rt}{\right}

\newcommand{\pa}{\partial}

\newcommand{\md}{\mathcal{D}}

\newcommand{\sfI}{\mathsf{I}}
\newcommand{\sfJ}{\mathsf{J}}
\newcommand{\sfK}{\mathsf{K}}
\newcommand{\sfL}{\mathsf{L}}

\newcommand{\mm}{\mathcal M}

\newcommand{\intr}{\int_{\R^d}}

\stackMath

\newcommand\rwhat[1]{%
\savestack{\tmpbox}{\stretchto{%
  \scaleto{%
    \scalerel*[\widthof{\ensuremath{#1}}]{\kern-.6pt\bigwedge\kern-.6pt}%
    {\rule[-\textheight/2]{1ex}{\textheight}}
  }{\textheight}%
}{0.5ex}}%
\stackon[1pt]{#1}{\tmpbox}%
}
\begin{document}
\allowdisplaybreaks

\date{\today}


\begin{abstract} In this paper, we present a refined framework for the global-in-time well-posedness theory for the pressureless Euler--Navier--Stokes system and the optimal temporal decay rates of certain norms of solutions. Here the coupling of two systems, pressureless Euler system and incompressible Navier--Stoke system, is through the drag force. We construct the global-in-time existence and uniqueness of regular solutions for the pressureless Euler--Navier--Stokes system without using a priori large time behavior estimates. Moreover, we seek for the optimal Sobolev regularity for the solutions. Concerning the temporal decay for solutions, we show that the fluid velocities exhibit the same decay rate as that of the heat equations. In particular, our result provides that the temporal decay rate of difference between two velocities, which is faster than the fluid velocities themselves, are at least the same as the second-order derivatives of fluid velocities. 
\end{abstract}

\maketitle \centerline{\date}


%
%
%
%
\section{Introduction}\label{sec:1}
\setcounter{equation}{0}
In this work, we revisit the global-in-time well-posedness and the large time behavior estimates for a coupled hydrodynamic system in the whole space. Our main system consists of the pressureless Euler equations and incompressible Navier--Stokes equations (in short, pressureless ENS system), which are coupled through the drag force, also often called friction force. More precisely, let $\rho = \rho(x,t)$ and $u = u(x,t)$ be the density and velocity for the pressureless Euler fluid flow, respectively, and let $v = v(x,t)$ be the velocity for the incompressible Navier--Stokes fluid flow. Then our main system reads as
\begin{align}
\begin{aligned}\label{A-1}
&\partial_t \rho + \nabla_x \cdot (\rho u) = 0,\quad x \in \R^d, \ t > 0,\\
&\partial_t (\rho u) + \nabla_x \cdot (\rho u \otimes u) = -\rho( u-v),\\
&\partial_t v + (v \cdot\nabla_x )v + \nabla_x p -\Delta_x v = \rho(u-v),\\
&\nabla_x \cdot v =0,
\end{aligned}
\end{align}
subject to initial data:
\bq\label{A-1_ini}
(\rho(x,0), u(x,0), v(x,0)) = (\rho_0(x), u_0(x), v_0(x)), \quad x\in\R^d.
\eq

The system \eqref{A-1} can be formally derived from the Vlasov--Navier--Stokes system, which describes the behavior of a large cloud of particles interacting with the incompressible fluid \cite{OR81, Wi58}, in the case of mono-kinetic particle distributions, see \cite{CJ21} for more details. On the other hand, if the pressureless Euler equations in \eqref{A-1} are replaced by the isothermal Euler equations, i.e., the pressure term $\nabla_x \rho$ is added, then the resulting system, Euler--Navier--Stokes system, can be rigorously derived from Vlasov--Fokker--Planck--Navier--Stokes system by considering a singular parameter in the nonlinear Fokker--Planck operator \cite{CCK16}. For the interactions with compressible fluid, the derivation of Euler--Navier--Stokes system, can also be rigorously derived \cite{CJpre}. We also refer to \cite{Gou01, GJV04, GJV04_2, MV08} for the hydrodynamic limits for Vlasov--Navier--Stokes system in various regimes. 

The main purpose of the current work is to improve the previous work \cite{CJ21}, where the global Cauchy problem for the system \eqref{A-1} is discussed. Under suitable smallness and regularity assumptions on the initial data, the global-in-time existence and uniqueness of classical solutions to the system \eqref{A-1} are established in \cite{CJ21} when $d\geq 3$. Since the pressureless Euler equations in \eqref{A-1} may develop a finite-time formation of singularities, in order to prevent the possible finite-time breakdown of smoothness of solutions, {\it a priori} large time behavior estimate of solutions is appropriately used in \cite{CJ21}, see also \cite{CK16, HKK14} for the case of periodic domain. In particular, the large time behavior estimate gives the time-integrability of $\|(\nabla_x u)(\cdot,t)\|_{H^s}$, and thus this combined with classical energy estimates provides the uniform-in-time bound on $\|\rho(\cdot,t)\|_{H^s}$. To be more specific, the following temporal decay estimate of solutions is obtained in \cite{CJ21}:
\bq\label{intro1}
\|u(\cdot,t)\|_{H^{s+1}(\R^d)}^2 + \|v(\cdot,t)\|_{H^s(\R^d)}^2 \leq \frac{C}{(1+t)^{\alpha}} \quad \forall\,t>0,
\eq
for any $\alpha \in (0,d/2)$ and some $C>0$ independent of $t$. Then we find
\[
\int_0^t \|(\nabla_x u)(\cdot,\tau)\|_{H^s(\R^d)}\,d\tau \leq \lt(\int_0^t \|(1+\tau)^\alpha(\nabla_x u)(\cdot,\tau)\|_{H^s(\R^d)}^2\,d\tau \rt)^{1/2} \lt(\int_0^t (1+\tau)^{-2\alpha}\,d\tau \rt)^{1/2},
\]
and the first term on the right-hand side of the above is bounded uniformly in time due to the decay estimate \eqref{intro1}. The second term is uniformly bounded if $\alpha > 1$, and thus this imposes the condition on the dimension $d >2$. It is worth noticing that we cannot obtain decay estimates for the density $\rho$ in time due to the absence of pressure. In the presence of pressure, i.e., Euler--Navier--Stokes system, as established in \cite{C15, C16}, it is not necessarily required to use the a priori estimate of large time behavior of solutions to have the global classical solutions. 

Main contributions of the present work are two-fold. First, we establish a refined framework for the global-in-time existence and uniqueness of regular solutions for the system \eqref{A-1} without employing {a priori} large time behavior estimates for velocities $u$ and $v$. We also seek for the optimal Sobolev regularity for the triplet $(\rho, u,v)$. In \cite{CJ21}, the initial data $(\rho_0, u_0, v_0)$ lie in $H^s(\R^d)\times H^{s+2}(\R^d)\times H^{s+1}(\R^d)$ with $s \ge 2[d/2]+1$, where $[\,\cdot\,]$ stands for the floor function. We relax this regularity condition so that the initial data lies in $H^s(\R^d)\times H^m(\R^d) \times H^s(\R^d)$ where $m>d/2+1$ and $s \in (m-2,m-1]$. Note that $s$ can be chosen as $s = d/2 - 1 + \epsilon$ for any sufficiently small $\epsilon > 0$, which is almost the critical $L^2$-regularity for the global well-posedness of the incompressible Navier-Stokes system \cite{FK64}. Moreover, $\rho$ and $\nabla v$ may not be bounded with this regularity. We would also like to point out that $L^\infty$ bound on the gradient of $u$ appears in the usual $L^p$ estimate of $\rho$, and hence our solution space for $u$ is also quite natural due to the classical Sobolev inequality $\|\nabla u\|_{L^\infty} \leq C\|\nabla u \|_{H^{m-1}}$ with $m > d/2 + 1$. Thus we can say that the solution we constructed has almost the critical $L^2$ regularity if $m$ and $s$ are chosen very close to $\frac d2+1$ and $m-2$, respectively. Furthermore, we require lower Sobolev regularities of solutions compared to the result of \cite{CJ21} to establish the global-in-time well-posedness for the system \eqref{A-1}.  For the critical regularity in Besov spaces, we refer to \cite{ZCLZ23}. 

For the desired estimates, we fully use the dissipative effect from the viscosity in the Navier--Stokes equations to have the desired regularity of solutions to the pressureless Euler equations through the drag forcing term, which is the coupling term. Precisely, by transforming the commutator estimate to suit our setting, we control $H^s$-norm of $\rho$ by time integrals of some $L^p$ norms of derivatives of $u$ and $v$ and $u-v$. For higher-order derivatives of $u$, we consider the drag forcing term, which plays a role as a relative damping, in the pressureless Euler equations as the linear damping in $u$ and have the dissipation rate for $u$. On the other hand, we make use of a delicate analysis of several Sobolev norms of $v$ based on the mild form of solution via Fourier transform, together with a deep understanding of the structure of the coupled system, to control the aforementioned terms. A proper combination of those estimates enables us to have the uniform-in-time bound for solutions $(\rho, u, v)$ in $H^s(\R^d)\times H^m(\R^d)\times H^s(\R^d)$. In regard to the uniqueness of solutions in our solution space, the low regularity of solutions $\rho$ causes some difficulties in estimating the stability of solutions for the system \eqref{A-1} in the usual $L^2$ space. In order to handle it, we employ a negative Sobolev space for the estimate of $\rho$ and establish the stability of solutions, see Section \ref{ssec_sta} for details. Clearly, if $m$ is large enough so that $\rho$ belongs to a sufficiently regular space, then the stability of solutions $\rho$ can be obtained in the $L^2$ space, see Remark \ref{rmk_l2_sta}. 

Our second main result concerns a better decay rate of convergences of solutions compared to the previous work \cite{CJ21, GWZ23}. We first show that the estimate \eqref{intro1} can reach the case $\alpha = d/2$, and moreover, $u$ and $v$  exhibit a temporal decay of order $t^{-(\frac d4 +\frac k2)}$ in $\dot{H}^k$ when $k \le \min\{m, \frac d2+2\}$, equivalent to the decay rate of the heat equation. Furthermore, we find out that the difference $(u-v)$ exhibits a temporal decay of order $t^{-(\frac d4 +\frac\theta 2 +1)}$ in $\dot{H}^{\theta}$ when $\theta \le \min\{m-2, \frac d2\}$, i.e., it decays to zero as fast as $\Delta v$. The main idea of proof is a bit similar to that of the existence proof. We introduce Sobolev norms of velocities $u$, $v$, and relative velocity $u-v$ with polynomial weights with respect to time and provide some relations among them. The velocity $u$ can be estimated rather readily by taking into account the drag force as the linear damping. Then again, we analyze the velocity $v$ of the Navier-Stokes equations elaborately to have uniform-in-time bound estimates of the aforementioned weighted Sobolev norms. This finally provides the optimal temporal decay rate of solutions, obviously improving the previous result \eqref{intro1}.

%
%
%
\subsection{Notation} For a function $f=f(x)$, $k \in \R$, and $p\in[1,\infty]$, $\|f\|_{W^{k,p}}$ represents the usual $W^{k,p}(\R^d)$-norm, in particular, $H^k(\R^d) = W^{k,2}(\R^d)$. For $\ell \in \R$, we denote $\pa^\ell$ by any partial derivative of order $\ell$ and $\nabla^\ell$ by the $\ell$-th order total derivative. For simplicity, we omit $x$-dependence of differential operators, i.e., $\nabla f:= \nabla_x f$ and $\Delta f := \Delta_x f$. We denote by $C$ a generic positive constant depending only on the norms of the data, but independent of $t$. Finally, $f \ls g$ represents that there exists a positive constant $C>0$ such that $f \leq Cg$.
%
%
%
\subsection{Main result}
We state our main theorems on the global-in-time existence and uniqueness of regular solutions to the system \eqref{A-1} and the optimal temporal decay. 

\begin{theorem}[Global-in-time well-posedness]\label{main_thm1}
	Let $d \in \N$ with $d \geq 2$, $m>\frac {d}{2} + 1$, and $s \in (m-2,m-1]$. Suppose that the initial data $(\rho_0, u_0, v_0)$ satisfy
\begin{itemize}
\item[(i)] $\rho_0(x)>0$ for every $x\in\R^d$ and 
\item[(ii)] $(\rho_0, u_0, v_0) \in (L^1\cap H^s)(\R^d) \times H^m(\R^d) \times (L^1 \cap H^s)(\R^d)$.
\end{itemize}
If
\[
\| \rho_0 \|_{H^s} + \| u_0 \|_{H^m} + \| v_0 \|_{H^{s} \cap L^1} < \e_0
\]
for $\e>0$ sufficiently small,  the Cauchy problem \eqref{A-1}-\eqref{A-1_ini} has a unique global classical solution	
\[
(\rho, u, v) \in \mc([0,+\infty);H^s(\R^d)) \times \mathcal{C}([0,+\infty);H^m(\R^d)) \times \mathcal{C}([0,+\infty);H^s(\R^d)).
\]
Moreover, we have
\[
\| \nabla^2 u \|_{L^{1}(0,\infty;H^{m-2})}  + \| \nabla^2 v \|_{L^{1}(0,\infty;H^{m-2})} + \| \nabla u\|_{L^{1}(0,\infty;L^{\infty})} + \| \nabla v \|_{L^{1}(0,\infty;L^{\infty})} + \| \nabla v \|_{L^2(0,\infty;H^{s})} <\infty.
\]
\end{theorem}

\begin{remark}
The solution space of $u$ covers all subcritical Sobolev spaces of Euler equations. We also remark that the $L^1$-condition for $v_0$ can be replaced by $L^q$ with any $q \in [1,2)$ in Theorem \ref{main_thm1}. In that case, we need to adjust corresponding estimates in Section \ref{sec:3}. 
\end{remark}

 Next, we discuss the temporal decay estimates.

\begin{theorem}[Temporal decay estimates]\label{main_thm2}
Let the assumptions of Theorem \ref{main_thm1} be satisfied. For $k \in [0,\min\{\frac d2 +2, m\}]$ and $\theta\in [0,\min\{ \frac d2, m-2\}]$, we have
\[
\| \nabla^k u(t)\|_{L^2} + \| \nabla^k  v(t)\|_{L^2}  \le Ct^{-\frac{d}{4}-\frac k2} 
\]
and
\[
\|\nabla^\theta(u-v)(t)\|_{L^2} \le Ct^{-\frac{d}{4}-\frac{\theta}{2}-1}
\]
for $t > 0$. Furthermore, if $m\ge \frac d2+2$ we  also have
\[
 \|\Delta v(t)\|_{L^\infty} + \|(u-v)(t)\|_{L^\infty} \le Ct^{-\frac d2 -1},
\]
and if $m>\frac d2+2$ or  $\| \Delta  u_0\|_{L^\infty}<\infty$ when $m=\frac d2+2$, 
\[
\|\Delta u(t)\|_{L^\infty} \le Ct^{-\frac d2 -1}.
\]
Moreover, for $\ell \in [\frac d2+2, m]$,
\[
\| \nabla^\ell u(t)\|_{L^2} + \| \nabla^\ell  v(t)\|_{L^2} + \|\nabla^{\ell-2}(u-v)(t)\|_{L^2}  \le Ct^{-\frac{d}{2}- 1}  \quad \mbox{for }\ t>0.
\]
\end{theorem}

\begin{remark}
Here we provide some remarks on the results in Theorem \ref{main_thm2}.
\begin{enumerate}
\item[(i)] 
Note that we obtained temporal $L^2$ decay of $v$ up to the same regularity as $u$. Moreover, since the decay rate for $v$  corresponds to the large time behavior of the heat equation, we can say it is an optimal rate. As stated before, we cannot have any decay estimates for $\rho$ since there is no dissipative effect on the density.

\item[(ii)]
The decay rate of $u-v$ is faster than  $u$ and $v$ and the same as that of $\Delta v$, which is a much faster rate than previous results \cite{CJ21, GWZ23}.

\item[(iii)]
For $\theta \in [0, \min\{d/2,m-2\}]$ with $d/2-\theta \le s$, one can get
\[
\|\rho(u-v)\|_{L^2} \le Ct^{-\frac{d}{4}-\frac{\theta}{2}-1},
\]
and hence the drag force $\rho(u-v)$ exhibits a faster temporal decay than the relaitve velocity $(u-v)$ in $L^2(\R^d)$.

\end{enumerate}
	\end{remark}

%
%
%
\subsection{Organization of paper}
The rest of this paper is organized as follows. In Section \ref{sec:pre}, we present several inequalities used throughout this paper, large time behavior estimates of the total energy, and local-in-time existence and uniqueness of solutions for system \eqref{A-1} in our desired Sobolev spaces. In Section \ref{sec:3}, we provide the {\it a priori} estimates for the regular solutions and establish the global existence result resulting in the proof of Theorem \ref{main_thm1}. In Section \ref{sec:4}, we establish the temporal decay estimates of solutions obtained in Theorem \ref{main_thm1} proving Theorem \ref{main_thm2}.

%
%
%
\section{Preliminaries}\label{sec:pre}
In this section, we present several useful Sobolev and interpolation inequalities and time decay estimate of the total energy for the pressureless ENS system \eqref{A-1} that will be significantly and frequently used throughout the paper. We also state the local-in-time well-posedness for the system \eqref{A-1}.

\subsection{Useful inequalities}

We first list some classical inequalities in the lemma below.
\begin{lemma}\label{lem_moser} 
\begin{itemize}
\item[(i)] For any pair of functions $f,g \in (H^k \cap L^\infty)(\R^d)$, we obtain
\[
\|\nabla^k (fg)\|_{L^2} \le C\lt(\|f\|_{L^\infty} \|\nabla^k g\|_{L^2} + \|\nabla^k f\|_{L^2}\|g\|_{L^\infty}\rt).
\]
Furthermore, if $\nabla f \in L^\infty(\R^d)$, we have
\[
\|\nabla^k(fg) - f\nabla^k g\|_{L^2} \le C\lt(\|\nabla f\|_{L^\infty}\|\nabla^{k-1} g\|_{L^2} + \|g\|_{L^\infty}\|\nabla^k f\|_{L^2}\rt).
\]
Here $C>0$ only depends on $k$ and $d$.

\item[(ii)] For $f \in H^{[d/2]+1}(\R^d)$ with $d\ge 3$, we have
\[
\|f\|_{L^\infty} \leq C\|\nabla f\|_{H^{[d/2]}}.
\]
\item[(iii)] For $ 0\leq \ell \leq (d-1)/p$ and $f \in W^{p,\ell}(\R^d)$, we have
\[
\|f\|_{L^{\frac{1}{\frac1p-\frac{\ell}{d}}}} \le C \|\nabla^\ell f\|_{L^p}.
\]
\item[(iv)] For $s$, $s_1$ and $s_2$ satisfying $s \le \min\{ s_1, s_2, s_1+s_2-d/2\}$ and $(f,g) \in H^{s_1}(\R^d) \times H^{s_2}(\R^d)$,
\[
\|fg\|_{H^s} \le C\|f\|_{H^{s_1}}\|g\|_{H^{s_2}}.
\]
\item[(v)] Let $q,r$ be any numbers satisfying $1 \leq q,r \leq \infty$, and let integers $j,m$ satisfy $0 \leq j \leq m$. If $f \in (W^{m,r} \cap L^q)(\R^d)$, then
\[
\|\nabla^j f\|_{L^p} \leq C\|\nabla^m f\|_{L^r}^\alpha \|f\|_{L^q}^{1-\alpha},
\]
where 
\[
\frac1p = \frac jd + \alpha\lt(\frac1r - \frac md \rt) + (1-\alpha) \frac1q
\]
for all $\alpha$ with
\[
\frac jm \leq \alpha \leq 1.
\]
Here $C>0$ is independent of $f$.

\item[(vi)]\cite{GO14} Let $\frac 12 < r<\infty$ and $1 <p_1, p_2, q_1, q_2 \le \infty$ satisfy $\frac1r = \frac{1}{p_1} + \frac{1}{q_1} = \frac{1}{p_2} + \frac{1}{q_2}$. Given $s > \max\{0, \frac{d}{r}-d\}$ or $s \in 2\N$, there exists $C= C(d,s,r,p_1, p_2, q_1, q_2)$ such that for all $f, g \in \mathscr{S}(\R^d)$, where $\mathscr{S}(\R^d)$ denotes the Schwartz space,
\[
\|\nabla^s (fg)\|_{L^r} \le C\lt( \|\nabla^s f\|_{L^{p_1}}\|g\|_{L^{q_1}} + \|f\|_{L^{p_2}}\|\nabla^s g\|_{L^{q_2}}\rt).
\]
\item[(vii)]\cite{JKLpre} For $p,q>0$, there exists $C=C(p,q)>0$ satisfying
\[
\| f\|_{L^1} \le C \lt\||x|^{\frac d2 + p} f\rt\|_{L^2}^{\frac{q}{p+q}} \lt\| |x|^{\frac d2-q}f \rt\|_{L^2}^{\frac{p}{p+q}}.
\]
\end{itemize}
\end{lemma}

\subsection{Energy estimates \& local well-posedness} The total energy for the system \eqref{A-1} is given by
\[
E(t):=   \frac12\intr \rho(x,t)|u(x,t)|^2\,dx + \frac12\intr |v(x,t)|^2\,dx.
\]
Then it can be easily checked that $E$ satisfies 
\bq\label{energy_est}
\frac{d}{dt} E(t) + D(t) = 0,
\eq
under the sufficient regularity assumptions on the solutions, where $D$ is the dissipation rate given by 
\[
D(t):= \intr |\nabla v(x,t)|^2\,dx + \intr \rho(x,t)|(u-v)(x,t)|^2\,dx.
\]
The above estimate only gives that the total energy is not increasing in time, however, by using almost the same argument as in \cite[Proposition 3.1]{CJ21}, we can have the following estimate of the large time behavior of classical solutions to the system \eqref{A-1}.
\begin{proposition}\label{large_time_0}
For $T>0$ and $d\ge 2$, let $(\rho,u,v)$ be a classical solution to the pressureless ENS system \eqref{A-1} on the time interval $[0,T]$ satisfying $\|\rho\|_{L^\infty(\R^d \times (0,T))}<\infty$. Then, there exists a constant $C>0$ independent of $T$ such that for every $\alpha \in (0,d/2)$,
\[
E(t)(1+t)^\alpha + \int_0^t (1+\tau)^\alpha D(\tau)\,d\tau \le C\lt(E(0)+\|v_0\|_{L^1}^2\rt)
\]
for all $t \in [0,T]$.
\end{proposition}

Finally, in the theorem below, we give the local-in-time existence and uniqueness of classical solutions to the system \eqref{A-1}.  One may deduce its proof from a slight modification of \cite{CK16,HKK14} according to the arguments in this paper.
\begin{theorem}\label{thm_local} Let $d \ge 2$, $m > d/2+1$ and $s\in (m-2, m-1]$. Suppose that the initial data $(\rho_0, u_0, v_0)$ satisfy the assumptions (i) and (ii) in Theorem \ref{main_thm1}. Then for any positive constants $N < M$, there exists a positive constant $T_0$ depending only on $N$ and $M$ such that if 
\[
 \|\rho_0\|_{H^s} + \|u_0\|_{H^m} + \|v_0\|_{H^s} + \|v_0\|_{L^1} < N, 
 \]
then the pressureless ENS system \eqref{A-1}-\eqref{A-1_ini} admits a unique solution 
\[
(\rho, u, v) \in \mc([0,T_0];H^s(\R^d)) \times \mathcal{C}([0,T_0];H^m(\R^d)) \times \mathcal{C}([0,T_0];H^s(\R^d))
\] 
satisfying $\rho(x,t) > 0$ for all $(x,t) \in \R^d \times [0,T_0]$,
\[
\sup_{0 \leq t \leq T_0}\lt(  \|u(\cdot,t)\|_{H^m} + \|v(\cdot,t)\|_{H^s} \rt) \leq M, \quad \mbox{and} \quad \sup_{0 \leq t \leq T_0} \|\rho(\cdot,t)\|_{H^s}<\infty.
\]
\end{theorem}
%
%
%

\section{A priori estimates}\label{sec:3}

In this section, we provide the a priori estimates for the derivatives of $\rho$, $u$, and $v$, which will be used to derive our desired global-in-time existence estimates. To this end, let $T>0$, $d \ge2$, $m > d/2+1$, and $s \in (m-2, m-1]$ and define
\[
\mathfrak{X}(s;T) := \sup_{0 \le t \le T}\lt( \|\rho(\cdot,t)\|_{H^s}^2 + \|u(\cdot,t)\|_{H^m}^2 + \|v(\cdot,t)\|_{H^s}^2\rt) + \|v\|_{L^{\frac{d+4}{2}}(0,T;L^2)}^2< \e_1^2 \ll 1,
\]
\[
\mathfrak{X}^*(s;T) := \sup_{0 \le t \le T}\lt(\|u(\cdot,t)\|_{H^m}^2 + \|v(\cdot,t)\|_{H^s}^2\rt) + \|v\|_{L^{\frac{d+4}{2}}(0,T;L^2)}^2,
\]
\begin{align*}
\mathfrak{D}(s;T) &:= \int_0^T \|\nabla u(t)\|_{L^\infty}\,dt + \int_0^T \||\xi| \hat v(t)\|_{L^1_\xi}\,dt + \int_0^T \|\nabla^2 u(t)\|_{H^{m-2}}\,dt + \int_0^T \|\nabla^2 v(t)\|_{H^{m-2}}\,dt\\
&\quad + \int_0^T \|(u-v)(t)\|_{L^2}\,dt + \|v\|_{L^{\frac{d+4}{2}}(0,T;L^2)},
\end{align*}
and
\[
\mathfrak{X}_0(s) := \|\rho_0\|_{H^s}^2 + \|u_0\|_{H^m}^2 + \|v_0\|_{H^s}^2.
\]
Throughout this section, we assume that for $\e_1>0$ small enough,
\[
\mathfrak{X}(s;T) < \e_1^2 \ll 1.
\]
The idea of proof is mainly organized into three steps. 
\begin{itemize}
\item {\bf (Step I)} We first show that there exists $C>0$ independent of $T$ such that
\[
\sup_{0 \leq t \leq T} \|\rho(t)\|_{H^s} \leq \|\rho_0\|_{H^s} \exp\lt(C\mathfrak{D}(s;T) \rt)
\]
and
\[
\mathfrak{X}^*(s;T) \leq C(\mathfrak{X}_0(s) + \|v_0\|_{L^1}^2) + C\e_1 \mathfrak{D}(s;T) + C\e_1 \mathfrak{D}(s;T)^2.
\]

\item {\bf (Step II)} We next show that here exists $C>0$ independent of $T$ such that
\[
\mathfrak{D}(s;T) \le C\lt(\sqrt{\mathfrak{X}_0(s)} + \|v_0\|_{L^1} \rt) + C\mathfrak{X}_0(s).
\]

\item {\bf (Step III)} Finally, we use the bound estimates obtained in {\bf Steps I \& II} together with a standard continuation argument to construct the global solution. 
\end{itemize}

\subsection{Estimates of $\mathfrak{X}(s;T)$}
First, we begin with the $H^s$-estimate of $\rho$.
\begin{lemma}\label{rho_est_up}
For $T>0$, there exists $C>0$ independent of $T>0$ such that
\[
\sup_{0 \leq t \leq T} \|\rho(t)\|_{H^s} \leq \|\rho_0\|_{H^s} \exp\lt(C\mathfrak{D}(s;T) \rt).
\]
\end{lemma}
\begin{proof}
First, a direct computation gives  
\[
		\frac{d}{dt} \| \rho \|_{L^2}^2 \leq \| \nabla \cdot u \|_{L^{\infty}} \| \rho \|_{L^2}^2.
\]
For $\dot{H}^s$-estimate with $m-2 < s \leq m-1$, we use $||\xi|^{s+1} - |\eta|^{s+1} |\le  C|\xi-\eta|\lt( |\xi-\eta|^s + |\eta|^s\rt) $ to get
\[\begin{aligned}
 &\lt||\xi|^{s+1} \intr\hat f(\xi-\eta)  \hat g(\eta)\,d\eta - \intr \hat f(\xi-\eta)  |\eta |^{s+1} \hat g(\eta )\,d\eta\rt|\\
 \quad&\le C\intr |\xi-\eta|^{s+1}|\hat f(\xi-\eta)| |\hat g(\eta)|\,d\eta + C\intr |\xi-\eta|\hat f(\xi-\eta) |\eta|^s |\hat g(\eta)|\,d\eta\\
 \quad &\le C \lt(|\cdot|\mathds{1}_{ \{|\cdot| \le 1\}}\hat f\rt)\ast \hat g(\xi) +  C \lt(|\cdot|^{s+1}\mathds{1}_{ \{|\cdot| >1\} }\hat f\rt)\ast \hat g(\xi) + C \lt[(|\cdot|\hat f)\ast (|\cdot|^s \hat g)\rt](\xi).
\end{aligned}\]
This implies
\bq\label{comm_est_s}
\begin{aligned}
\| \nabla^{s+1} (fg) - f\nabla^{s+1} g\|_{L^2} &\le C\lt(\||\xi| \hat f\|_{L_\xi^1}\|g\|_{H^s} + \||\xi|^{s+1}\mathds{1}_{\{|\xi|>1\}}\hat f\|_{L_\xi^p}\| \hat g\|_{L_\xi^q}\rt)\\
&\le  C\lt(\||\xi| \hat f\|_{L_\xi^1} + \|\nabla^m f\|_{L^2}\rt)\|g\|_{H^s},
\end{aligned}
\eq
where $p$ and $q$ satisfy
\[
\frac 1p + \frac1q = \frac 32, \quad \frac1p \in \lt( \frac12, \ \frac12 + \frac{m-(s+1)}{d} \rt)\quad \mbox{and}\quad \frac1q \in \lt(\frac12, \ \frac12 + \frac sd\rt).
\]
By using the above commutator estimate, we obtain
\[\begin{aligned}
\frac12\frac{d}{dt}\|\nabla^s \rho\|_{L^2}^2 &= -\intr \lt(\nabla^s \nabla\rho \cdot u\rt)\nabla^s \rho\,dx - \intr \lt(\nabla^s(\nabla\cdot (\rho u)) -u \cdot \nabla^s \nabla\rho \rt)\nabla^s \rho\,dx\\
&\le C \lt( \||\xi| \hat u\|_{L_\xi^1} + \|\nabla^2 u\|_{H^{m-2}}\rt) \|\rho(t)\|_{H^s}^2.
\end{aligned}\]
Here, we use Lemma \ref{lem_moser} (vii) to get
\[
\||\xi| \hat u\|_{L_\xi^1} \le \| |\xi| \rwhat{(u-v)}\|_{L_\xi^1} + \| |\xi|\hat v\|_{L_\xi^1} \le C\|\nabla^m (u-v)\|_{L^2}^{\frac{\frac d2 +1}{m}}\|(u-v)\|_{L^2}^{\frac{m-\frac d2-1}{m}} +  \| |\xi|\hat v\|_{L_\xi^1},
\]
and this shows
\[
\frac12\frac{d}{dt} \|\rho(t)\|_{H^s}^2 \leq C \lt( \||\xi| \hat v\|_{L_\xi^1} + \|\nabla^2 u\|_{H^{m-2}}+ \|\nabla^2 v\|_{H^{m-2}} + \|(u-v)\|_{L^2}\rt) \|\rho(t)\|_{H^s}^2.
\]
Applying Gr\"onwall's lemma yields the desired result.
\end{proof}

Next, we estimate $u$ and $v$ separately.

\begin{lemma}\label{u_Hm_est}
For $T>0$, there exists $C>0$ independent of $T$ such that for $\ell =1,2$,
\[\begin{aligned}
&\sup_{t \in [0,T]} \|u(t)\|_{H^m}^\ell + \int_0^T \|\nabla^2 u(t)\|_{H^{m-2}}^\ell\,dt \cr
&\quad \le C\sqrt{\mathfrak{X}_0(s)}^\ell + C\e_1^\ell \int_0^T \|\nabla u(t)\|_{L^\infty}\,dt   + C \e_1^{\ell-1}\int_0^T \|(u-v)(t)\|_{L^2}\,dt +C\e_1^{\ell-1} \int_0^T \|\nabla^2 v(t)\|_{H^{m-2}} \,dt.
\end{aligned}\]
In particular, we have
\[
\sup_{t \in [0,T]} \|u(t)\|_{H^m}^2 + \int_0^T \|\nabla^2 u(t)\|_{H^{m-2}}^2\,dt \le C\mathfrak{X}_0(s) + C\e_1\mathfrak{D}(s;T).
\]
\end{lemma}
\begin{proof}
From the equations for $u$ in \eqref{A-1}, we have
\[
\frac12 \frac{d}{dt}\|u\|_{L^2}^2 = \frac12 \intr (\nabla\cdot u) |u|^2\,dx - \intr (u-v)\cdot u\,dx \le \|\nabla u\|_{L^\infty}\|u\|_{L^2}^2 + \|(u-v)\|_{L^2}\|u\|_{L^2}.
\]
For $k \in [2, m]$, we  use Lemma \ref{lem_moser} (i) to get
\begin{equation*}
	\begin{aligned}
		&\frac12 \frac{d}{dt}\|\nabla^k u\|_{L^2}^2 + \|\nabla^k u\|_{L^2}^2  \leq  C\|u\|_{H^m}\|\nabla^k u\|_{L^2}^2 + \|\nabla^k v\|_{L^2}\|\nabla^k u\|_{L^2}.
	\end{aligned}
\end{equation*}
Then we use Young's inequality, if necessary, and combine two results to get the desired result.
\end{proof}

\begin{lemma}\label{v_Hs_est}
For $T>0$, there exists $C>0$ independent of $T$ such that
\[
\sup_{t \in [0,T]}\|v(t)\|_{H^s}^2 + \int_0^T\|\nabla v(t)\|_{H^s}^2\,dt  \leq C\mathfrak{X}_0(s)  + C\e_1^2 \mathfrak{D}(s;T).
\]
\end{lemma}
\begin{proof}
It follows from \eqref{energy_est} that
\[
\begin{aligned}
		\sup_{t \in [0,T]}& \intr \rho |u|^2 \,d x + \sup_{t \in [0,T]} \intr |v|^2 \,d x + 2 \left(\int_0^T \intr\rho |u-v|^2 \,d x d t + \int_0^T \intr|\nabla v|^2 \,d x d t \right) \\
		&= \intr\rho_0 |u_0|^2 \,d x + \intr|v_0|^2 \,d x.
\end{aligned}
\]
On the other hand, we obtain from  the equations for $v$ in \eqref{A-1} that
\begin{align}\label{v_est}
	\begin{aligned}
		\frac {1}{2} \frac {d}{d t} \intr|\nabla^{s} v|^2 \,d x + \intr|\nabla^{s+1} v|^2 \,d x \leq \intr\nabla^{s}[(v \cdot \nabla)v] \cdot \nabla^{s}v \,d x + \intr|\nabla^{s-1} (\rho (u-v))||\nabla^{s+1} v| \,d x.
	\end{aligned}
\end{align}
Here, we use the commutator estimate \eqref{comm_est_s} to get
\[
	\begin{aligned}
		\left| \intr\nabla^{s}( (v \cdot \nabla)v) \cdot \nabla^{s}v \,d x \right| &= \left| \intr\lt[\nabla^{s}\nabla\cdot(v \otimes v) - v\cdot\nabla \nabla^s v \rt] \cdot \nabla^{s}v \,d x \right| \\
		&\le \|\nabla^{s}\nabla \cdot(v \otimes v) - v\cdot\nabla \nabla^s v \|_{L^2}\|\nabla^s v\|_{L^2}\\
		&\le C\lt( \||\xi|\hat v\|_{L_\xi^1} + \|\nabla^2 v\|_{H^{m-2}}\rt)\|\nabla^s v\|_{L^2}^2.
	\end{aligned}
\]
For the estimate of the second term on the right-hand side of \eqref{v_est}, we consider two cases: $s\geq 1$ and $s < 1$. When $s\ge 1$, we use Lemma \ref{lem_moser} (v) and (vi) to yield
\[
	\begin{aligned}
		\| \nabla^{s-1}(\rho(u-v)) \|_{L^2} &\leq C \| \nabla^{s-1} \rho \|_{L^{\frac {2d}{d-2(m-s-1)}}} \| u-v \|_{L^{\frac {d}{m-s-1}}} + C \| \rho \|_{L^p} \| \nabla^{s-1}(u-v) \|_{L^q} \\
		&\leq C \| \rho \|_{H^{m-2}} \| u-v \|_{L^{\frac {d}{m-s-1}}} + C \| \rho \|_{H^s} \| \nabla^{s-1+d/2-k}(u-v) \|_{L^2}  \\
		&\leq C \| \rho \|_{H^s} (\| u-v \|_{L^2} + \| \nabla^{s+1}u \|_{L^2} + \| \nabla^{s+1}v \|_{L^2}),
	\end{aligned}
\]
where $p$, $q$, and $k$ satisfy
\[
\frac d2-1 < k < \min\lt\{ \frac d2, s\rt\}, \quad \frac 1p = \frac12 - \frac kd, \quad \mbox{and}  \quad \frac1q = \frac12 - \frac{d/2-k}{d}.
\]
For $s< 1$, we estimate
\begin{equation*}
	\begin{aligned}
		\| \nabla^{s-1}(\rho(u-v)) \|_{L^2} &\leq C\| \rho(u-v) \|_{L^{\frac {2d}{d+2(1-s)}}} \\
		&\leq C \| \nabla^s \rho \|_{L^2} \| u-v \|_{L^{d}} \\
		&\leq C \| \rho \|_{H^{s}} (\| u-v \|_{L^2} + \| \nabla^{s+1}u \|_{L^2} + \| \nabla^{s+1}v \|_{L^2})
	\end{aligned}
\end{equation*}
due to Lemma \ref{lem_moser} (iii).

Hence, we use Young's inequality to obtain
\[
\begin{aligned}
		\sup_{t \in [0,T]} &\intr|\nabla^{s} v|^2 \,d x + \int_0^T \intr|\nabla^{s+1} v|^2 \,d x d t \\
		&\leq \intr|\nabla^{s} v_0|^2 \,d x + C \sup_{t \in [0,T]} \| v(t)\|_{H^{s}}^2 \int_0^T \|\nabla^2 v(t)\|_{H^{m-2}} \, d t + \int_0^T \|\nabla^{s-1}(\rho (u-v))(t)\|^2_{L^2} \,d t\\
		&\le \intr|\nabla^{s} v_0|^2 \,d x + C\e_1^2 \int_0^T \|\nabla^2 v(t)\|_{H^{m-2}} \, d t \\
		&\quad + C\e_1^3 \int_0^T \| (u-v)(t) \|_{L^2} \,d t + C\e_1^2 \int_0^T (\| \nabla^{s+1}u(t) \|_{L^2}^2 + \| \nabla^{s+1}v(t) \|_{L^2}^2) \,d t\\
		&\le \intr|\nabla^{s} v_0|^2 \,d x + C\e_1^2 \int_0^T \|\nabla^2 v(t)\|_{H^{m-2}} \, d t  + C\e_1^3 \int_0^T \| \nabla^m u(t) \|_{L^2}\,dt \\
		&\quad + C\e_1^3 \int_0^T \| (u-v)(t) \|_{L^2} \,d t + C\e_1^2 \int_0^T  \| \nabla^{s+1}v(t) \|_{L^2}^2 \,d t,\\
\end{aligned}
\]
where we used
\[
\|\nabla^{s+1} u\|_{L^2} \le \|\nabla^m u\|_{L^2}^{\frac{s+1}{m}}\|u\|_{L^2}^{\frac{m-s-1}{m}}, \quad \frac{s+1}{m} \in \lt(\frac12, 1\rt].
\]
Finally, we take $\e_1 > 0$ small enough to conclude the desired inequality.
\end{proof}

We finally provide the bound estimate of $v$ in $L^{\frac{d+4}{2}}(0,T;L^2)$.
\begin{lemma}\label{v_lp_est}
For $T>0$, there exists $C>0$ independent of $T$ such that 
\[
\|v\|_{L^{\frac{d+4}{2}}(0,T;L^2)} \leq C\lt(\sqrt{\mathfrak{X}_0(s)} + \|v_0\|_{L^1}\rt) + C\e_1 \mathfrak{D}(s;T).
\]
\end{lemma}
\begin{proof}
Taking the Fourier transform to the equations of $v$ in \eqref{A-1}, we find
\bq\label{v_trans}
	\begin{aligned}
		\partial_t \hat v - |\xi|^2 \hat v + P(\xi) \rwhat{(v \cdot \nabla)v} = P(\xi)\rwhat{\rho(u-v)},
	\end{aligned}
\eq
where $P(\xi) = \mathbb{I} - \frac{\xi \otimes \xi}{|\xi|^2}$ denotes the Leray projection. This gives
\[
	\begin{aligned}
		\hat v(t) = e^{-|\xi|^2t} \hat v_0 - \int_0^t e^{-|\xi|^2(t-\tau)} P(\xi) \rwhat{(v \cdot \nabla)v}(\tau) \,d \tau + \int_0^t e^{-|\xi|^2(t-\tau)} P(\xi) \rwhat{\rho(u-v)}(\tau) \,d \tau,
	\end{aligned}
\]
and from which, we obtain
\begin{equation*}
	\begin{aligned}
		\| \hat v \|_{L^p(0,T;L^2)} &\leq \| e^{-|\xi|^2t} \hat v_0 \|_{L^p(0,T;L^2)} + \left\| \int_0^t e^{-|\xi|^2(t-\tau)} P(\xi) \rwhat{ (v \cdot \nabla)v}(\tau) \,d  \tau \right\|_{L^p(0,T;L^2)} \\
		&\quad + \left\| \int_0^t e^{-|\xi|^2(t-\tau)} P(\xi)\rwhat{  \rho(u-v)} \,d  \tau \right\|_{L^p(0,T;L^2)}\\
		&=: \sfI_1 +\sfI_2 + \sfI_3.
	\end{aligned}
\end{equation*}
First, for $\sfI_1$, we use
\begin{equation*}
	\begin{aligned}
		\| e^{-|\xi|^2t} \hat v_0 \|_{L^2} \leq \| v_0 \|_{L^2} \quad \mbox{for} \quad t \in(0,1)
	\end{aligned}
\end{equation*}
and 
\begin{equation*}
	\begin{aligned}
		\| e^{-|\xi|^2t} \hat v_0 \|_{L^2} \leq \| \hat v_0 \|_{L^{\infty}} \| e^{-|\xi|^2t} \|_{L^2} \leq Ct^{-\frac {d}{4}} \| v_0 \|_{L^1}  \quad \mbox{for} \quad t \in [1,\infty)
	\end{aligned}
\end{equation*}
to have
\[
		\sfI_1 \leq C \lt(\| v_0 \|_{L^2} + \| v_0 \|_{L^1}\rt).
\]
For $\sfI_2$, since

\[		\left\|  e^{-|\xi|^2(t-\tau)} P(\xi) \rwhat{  (v \cdot \nabla)v} (\tau)\right\|_{L^2} \leq  \| v(\tau) \|_{L^2} \| \nabla v(\tau) \|_{L^{\infty}} 
\]
and
\[
	\begin{aligned}
		\left\| e^{-|\xi|^2(t-\tau)} P(\xi) \rwhat{  (v \cdot \nabla)v}(\tau)  \right\|_{L^2} &\leq C (t-\tau)^{-\frac {1}{2}} \| (v \otimes v)(\tau) \|_{L^{2}}  \\
		&\leq C (t-\tau)^{-\frac {1}{2}} \| v (\tau)\|_{L^2}^{\frac {d+4}{d+2}} \| \nabla^2 v(\tau) \|_{H^{m-2}}^{\frac {d}{d+2}}
	\end{aligned}
\]
hold, we use Minkowski's integral inequality and Young's convolution inequality to see
\[\begin{aligned}
		\sfI_2&\leq C \sup_{t \in [0,T]} \| v (t)\|_{L^2} \int_0^T \| \nabla v (t)\|_{L^{\infty}} \,d  t +C \int_0^T \| v(\tau) \|_{L^2}^{\frac {d+4}{d+2}} \| \nabla^2 v(\tau) \|_{H^{m-2}}^{\frac {d}{d+2}} \left\| (\cdot-\tau)^{-\frac {1}{2}} \right\|_{L^p(\tau+1,T)} d  \tau \\
		&\leq C \e_1 \int_0^T \| \nabla v(t) \|_{L^{\infty}} \,d  t +C \int_0^T \| v(t) \|_{L^2}^{\frac {d+4}{d+2}} \| \nabla^2 v (t)\|_{H^{m-2}}^{\frac {d}{d+2}}\, dt.
\end{aligned}\]
Note that the second term on the right-hand side of the above can be estimated as 
\begin{align}\label{est_vv}
\begin{aligned}
\int_0^T \| v(t) \|_{L^2}^{\frac {d+4}{d+2}} \| \nabla^2 v(t) \|_{H^{m-2}}^{\frac {d}{d+2}} \,d t 
&\leq \lt(\int_0^T \|v(t)\|_{L^2}^{\frac{d+4}{2}}\,dt\rt)^{\frac{2}{d+2}} \lt(\int_0^T \|\nabla^2 v(t)\|_{H^{m-2}}\,dt \rt)^{\frac{d}{d+2}}\cr
&\leq C\e_1 \|v\|_{L^p(0,T;L^2)}^{\frac{2}{d+2}} \lt(\int_0^T \|\nabla^2 v(t)\|_{H^{m-2}}\,dt \rt)^{\frac{d}{d+2}}\cr
&\leq  C\e_1 \|v\|_{L^p(0,T;L^2)} + C\e_1\int_0^T \|\nabla^2 v(t)\|_{H^{m-2}}\,dt,
\end{aligned}
\end{align}
due to $p = \frac {d+4}{2} > 2$. Thus, we obtain
\[
\sfI_2 \le C\e_1 \|v\|_{L^p(0,T;L^2)} + C\e_1\int_0^T\|\nabla^2 v(t)\|_{H^{m-2}}\,dt + C\e_1 \int_0^T \|\nabla v(t)\|_{L^\infty}\,dt.
\]
For $\sfI_3$, we estimate
\[
	\begin{aligned}
		\left\| e^{-|\xi|^2(t-\tau)} P(\xi)\rwhat{ \rho(u-v)} (\tau) \right\|_{L^2} &\leq C \| e^{-|\xi|^2 \frac {t-\tau}{2}} \|_{L^2} \| \rwhat {\rho (u-v)}(\tau) \|_{L^{\infty}}  \\
		&\leq C (t-\tau)^{-\frac {d}{4}} \| \rho(\tau) \|_{L^2} \| (u-v)(\tau) \|_{L^2}
	\end{aligned}
\]
and
\[
	\begin{aligned}
		\left\| e^{-|\xi|^2(t-\tau)} P(\xi)\rwhat{  \rho(u-v)} (\tau) \right\|_{L^2} & \leq  \| \rho(u-v)(\tau) \|_{L^2}  \\
		&\leq \|\rho(\tau)\|_{L^2}\|(u-v)(\tau)\|_{L^\infty}\\
		&\le C  \|\rho(\tau)\|_{H^s}\lt(\|(u-v)(\tau)\|_{L^2} + \|\nabla^2 u(\tau)\|_{H^{m-2}} + \|\nabla^2 v(\tau)\|_{H^{m-2}}\rt).
			\end{aligned}
\]
Hence, we have
\[\begin{aligned}
\sfI_3&\le C\e_1 \int_0^T \|(u-v)(t)\|_{L^2}\,dt + C\e_1 \int_0^T \|\nabla^2 u(t)\|_{H^{m-2}}\,dt + C\e_1 \int_0^T \|\nabla^2 v(t)\|_{H^{m-2}}\,dt.
\end{aligned}\]
Therefore, we gather all the estimates for $\sfI_i$'s and use the smallness of $\e_1$ to complete the proof.
\end{proof}

%
%
%
%
%

\subsection{Estimates of $\mathfrak{D}(s;T)$}
As stated in the previous subsection, it is clear that the uniform-in-$T$ estimates for 
\[\begin{aligned}
&\nabla u \in L^1(0,T;L^{\infty}), \quad  |\xi| \hat v \in L^1(0,T; L^1_\xi),\\
&\nabla^k v \in L^1(0,T;L^2)\ \mbox{ for }  k \in [2, m], \quad \mbox{and} \quad u-v \in L^1(0,T;L^2)
\end{aligned}\]
are necessary to get the uniform-in-$T$ bound for $\mathfrak{X}(s,T)$. So we proceed to the estimates for $u$. Above all, we investigate the Lipschitz estimate for $u$ and $v$.

\begin{lemma}\label{u_lip_est}
For $T>0$, there exists $C>0$ independent of $T>0$ such that
\[
\begin{aligned}
\int_0^T \|\nabla u(t)\|_{L^{\infty}} \,d t \leq C\| \nabla u_0\|_{L^\infty} + C\int_0^T  \| \nabla v(t) \|_{L^{\infty}} \,d t.
\end{aligned}
\]
\end{lemma}
\begin{proof}
From the $u$ equation, we have for any $p \in [2,\infty)$ that
\[
\frac {1}{p} \frac{d}{dt} \intr |\nabla u|^p \,dx + \intr |\nabla u|^p \,dx \le C\|\nabla u\|_{L^\infty}\|\nabla u\|_{L^p}^p + \intr |\nabla v||\nabla u|^{p-1} \,dx.
\]
Thus,
\begin{equation*}
	\begin{aligned}
		\frac{d}{dt} \|\nabla u\|_{L^p} + \frac12\|\nabla u\|_{L^p} \le C\| \nabla v \|_{L^p},
	\end{aligned}
\end{equation*}
and subsequently, Gr\"onwall's inequality implies
\begin{equation*}
	\begin{aligned}
		\|\nabla u(t)\|_{L^p} \leq e^{-\frac t2} \|\nabla u_0\|_{L^p}  + C\int_0^t e^{-\frac{(t-\tau)}2} \| \nabla v(\tau) \|_{L^p} \,d \tau,
	\end{aligned}
\end{equation*}
where $C>0$ is independent of $p$. Then tending $p \to \infty$ gives
\[
	\begin{aligned}
		\|\nabla u(t)\|_{L^{\infty}} \leq e^{-t} \|\nabla  u_0\|_{L^\infty}  +C\int_0^t e^{-(t-\tau)} \| \nabla v(\tau) \|_{L^{\infty}} \,d \tau,
	\end{aligned}
\]
and hence, we integrate the above relation with respect to $t$ on $[0,T]$ to get the desired inequality.
\end{proof}

\begin{lemma}\label{v_lip_est}
For $T>0$, there exists $C>0$ independent of $T>0$ such that
\[
\int_0^T  \| |\xi| \hat v(t) \|_{L^1_\xi}\,dt \leq C\lt( \sqrt{\mathfrak{X}_0(s)} + \|v_0\|_{L^1}\rt) + C\e_1 \mathfrak{D}(s;T) +  C\int_0^T \|\nabla v(t)\|_{H^s}^2\,dt .
\]
\end{lemma}
\begin{remark} Since $\| \nabla v \|_{L^{\infty}} \leq \| |\xi| \hat v \|_{L^1_\xi}$, Lemma \ref{v_lip_est} directly gives the bound estimate of $\| \nabla v \|_{L^1(0,T;L^{\infty})}$.
\end{remark}
\begin{proof}[Proof of Lemma \ref{v_lip_est}]
We use \eqref{v_trans} to estimate
\[\begin{aligned}
		\int_0^T \| |\xi| \hat v (t) \|_{L^1} \,d t &\leq \int_0^T \| |\xi| e^{-|\xi|^2t} \hat v_0 \|_{L^1} \,d t \\
		&\quad + \int_0^T \left\| \int_0^t |\xi| e^{-|\xi|^2(t-\tau)} P(\xi) \rwhat{(v \cdot \nabla)v} (\tau) \,d \tau \right\|_{L^1} d t \\
		&\quad + \int_0^T \left\| \int_0^t |\xi| e^{-|\xi|^2(t-\tau)} P(\xi) \rwhat{\rho(u-v)}(\tau) \,d \tau \right\|_{L^1} d t\\
		&=: \sfJ_1 + \sfJ_2+ \sfJ_3.
\end{aligned}\]
For $\sfJ_1$, one uses Fubini's theorem to attain
\begin{equation*}
	\begin{aligned}
		\sfJ_1 &=  \int_0^T \int_{B(0;1)} e^{-|\xi|^2t} |\xi| |\hat v_0 | \,d \xi d t + \int_0^T \int_{\R^d \setminus B(0;1)} e^{-|\xi|^2t} |\xi| |\hat v_0 | \,d \xi d t \\
		&\leq \| \hat  v_0 \|_{L^{\infty}} \int_{B(0;1)} \frac {1}{|\xi|} \lt(\int_0^T |\xi|^2 e^{-|\xi|^2t} \,dt\rt) d \xi + \int_{\R^d \setminus B(0,1)}\lt(\int_0^T |\xi|^2 e^{-|\xi|^2t} \,d t\rt)\, \frac {|\hat  v_0 |}{|\xi|^{1-\epsilon/2}} \,d \xi \\
		&\leq C\lt(\| v_0 \|_{L^1} + \| v_0 \|_{H^s}\rt),
	\end{aligned}
\end{equation*}
where $\epsilon>0$ satisfies $s>\frac {d}{2} -1 + \epsilon$.

For $\sfJ_2$, we may use $|P(\xi)| \leq 1$ and $(v \cdot \nabla)v = \nabla \cdot (v \otimes v)$ to get
\[
\sfJ_2\leq C\int_0^T \| \hat v(t) \|_{L^1}^2 \,d t.
\]
Since \cite[Lemma 2.1]{JKLpre} implies
\begin{equation}\label{intp_1}
	\begin{aligned}
		\| \hat  v \|_{L^1} \leq C \| \hat  v \|_{L^2}^{1-\frac {d}{2\ell}} \| \rwhat{  \nabla^\ell v} \|_{L^2}^{\frac {d}{2\ell}},
	\end{aligned}
\end{equation}
whenever $d/2\ell \in (0,1)$. Here we separately estimate depending on the dimension.

\vspace{.2cm}

\noindent $\diamond$ (Case (i): $d\ge 4$) Since $s+1 >m-1 > \frac d2 \ge 2$, when $ s+1< d$, we have 
\[
\frac{d}{s+1} = \lt(2-\frac{d}{s+1} \rt)\cdot 1 + \lt(\frac{d}{s+1}-1 \rt)\cdot 2\in (1,2)
\] 
and use Young's inequality to get
\[\begin{aligned}
\|\hat v\|_{L^1}^2 &\le C\|v\|_{L^2}^{2-\frac{d}{s+1}}\|\nabla^{s+1}v\|_{L^2}^{\frac{d}{s+1}} \\
&\le C\lt(\e_1\|\nabla^{s+1}v\|_{L^2}\rt)^{2-\frac{d}{s+1}}\|\nabla^{s+1}v\|_{L^2}^{2\lt(\frac{d}{s+1}-1\rt)} \\
&\le C\e_1 \|\nabla^2 v\|_{H^{m-2}} + C\|\nabla^{s+1} v\|_{L^2}^2.
\end{aligned}\]
When $s+1 \ge d$, the relation $m \ge s+1 \ge d$ gives
\[
\|\hat v\|_{L^1}^2 \le C\|v\|_{L^2}\|\nabla^d v\|_{L^2} \le C\e_1 \|\nabla^2 v\|_{H^{m-2}}.
\]
In either case, one has
\bq\label{Fv_est1}
\|\hat v\|_{L^1}^2 \le C\e_1\|\nabla^2 v\|_{H^{m-2}} + C\|\nabla^{s+1}v\|_{L^2}^2.
\eq

\noindent $\diamond$ (Case (ii): $d=3$) Here, when $s+1 \ge 2$, we again have \eqref{Fv_est1}. When $3/2 < s+1 <2$, we have $\frac{d}{s+1} \in \lt(\frac{2}{s+1}, 2 \rt)$ and one gets
\[
\frac{d}{s+1} = \frac{2-\frac{d}{s+1}}{2-\frac{2}{s+1}} \cdot \frac{2}{s+1} + \frac{\frac{d-2}{s+1}}{2-\frac{2}{s+1}}\cdot 2.
\]
Thus, we use
\[
\|\nabla^{s+1} v\|_{L^2} \le \|\nabla^2 v\|_{L^2}^{\frac{s+1}{2}}\| v\|_{L^2}^{1-\frac{s+1}{2}}
\]
and Young's inequality to get
\[\begin{aligned}
\|\hat v\|_{L^1}^2 &\le C\|v\|_{L^2}^{2-\frac{d}{s+1}}\|\nabla^{s+1} v\|_{L^2}^{\frac{d}{s+1}}\\
&\le C\e_1^{2-\frac{2}{s+1}}\|\nabla^{s+1} v\|_{L^2}^{\frac{2}{s+1}} +C \|\nabla^{s+1} v\|_{L^2}^2\\
&\le C \e_1^{2-\frac{2}{s+1}} \|v\|_{L^2}^{\frac{2}{s+1} -1} \|\nabla^2 v\|_{L^2} + C\|\nabla^{s+1} v\|_{L^2}^2 \\
&\le C \e_1\|\nabla^2 v\|_{H^{m-2}} + C\|\nabla^{s+1} v\|_{L^2}^2.
\end{aligned}
\]

\noindent $\diamond$ (Case (iii) $d=2$) Here, we just use
\[
\|\hat v\|_{L^1}^2 \le C\|v\|_{L^2}\|\nabla^2 v\|_{L^2}
\]
to get
\[
\|\hat v\|_{L^1}^2 \le C\e_1 \|\nabla^2 v\|_{H^{m-2}}.
\]
Thus, we gather the estimates from all the cases to yield
\[
\sfJ_2 \le C\e_1 \int_0^T\|\nabla^2 v(t)\|_{H^{m-2}}\,dt + C\int_0^T\|\nabla^{s+1} v(t)\|_{L^2}^2\,dt.
\]

For $\sfJ_3$, we estimate
\begin{equation*}
	\begin{aligned}
		\sfJ_3 &= \int_0^T \left\| \int_0^t |\xi| e^{-|\xi|^2(t-\tau)} P(\xi)\rwhat{  \rho(u-v)}(\tau) \,d \tau \right\|_{L^1}  d t \\
		&\quad \leq \int_0^T \int_{B(0;1)} \frac {1}{|\xi|} |\rwhat{  \rho(u-v)}(t)| \,d \xi d t + \int_0^T \int_{\R^d \setminus B(0;1)} \frac {1}{|\xi|} |\rwhat{\rho(u-v)}(t)| \,d \xi d t\\
		&= :\sfJ_{31} + \sfJ_{32}.
	\end{aligned}
\end{equation*}
Here $\sfJ_{31}$ can be bounded by
\begin{equation*}
	\begin{aligned}
		\int_0^T \| \rwhat {\rho(u-v)}(t) \|_{L^{\infty}} \,d t \left\| \frac {1}{|\xi|} \right\|_{L^1(B(0;1))} \leq C \sup_{t \in [0,T]} \| \rho(t) \|_{L^2} \int_0^T \|(u-v)(t) \|_{L^2} \,d t.
	\end{aligned}
\end{equation*}
For given $m > \frac {d}{2} + 1$, we can find $p = p(m) \in [1,2)$ such that 
\[
1-\frac {1}{d} < \frac {1}{p} < \frac {1}{2} + \frac {m-2}{d}. 
\]
Then, Young's convolution inequality and \eqref{intp_1} imply
\begin{equation*}
	\begin{aligned}
		\int_0^T \int_{\R^d \setminus B(0;1)} \frac {1}{|\xi|} |\rwhat{\rho(u-v)}(t)| \,d \xi d t 
		& \leq \left\| \frac {1}{|\xi|} \right\|_{L^{p'}(\R^d \setminus B(0;1))} \int_0^T \| \hat  \rho(t) \|_{L^p} \| \rwhat {(u-v)}(t) \|_{L^1} \,d t \\
		& \leq C \sup_{t \in [0,T]} \| \rho(t) \|_{H^{m-2}} \int_0^T \| (u-v)(t) \|_{H^m}\,d t\\
		& \le C\e_1 \int_0^T \lt(\|(u-v)(t)\|_{L^2}+  \|\nabla^2 u(t)\|_{H^{m-2}} + \|\nabla^2 v(t)\|_{H^{m-2}}\rt) dt.
	\end{aligned}
\end{equation*}
Thus, 
\[
\sfJ_3 \leq C\e_1 \int_0^T \lt(\|(u-v)(t)\|_{L^2}+  \|\nabla^2 u(t)\|_{H^{m-2}} + \|\nabla^2 v(t)\|_{H^{m-2}}\rt) dt.
\]
Finally, we gather all the estimates for $\sfJ_i$'s to get the desired relation.
\end{proof}

Next, we estimate $\nabla^2 v$ in $L^1(0,T; H^{m-2})$.

\begin{lemma}\label{v_Hm_est}
For $T>0$, there exists $C>0$ independent of $T$ such that
\[
\int_0^T \|\nabla^2 v(t)\|_{H^{m-2}}\,dt \le C\lt( \sqrt{\mathfrak{X}_0(s)} + \|v_0\|_{L^1}\rt) +  C\e_1 \mathfrak{D}(s;T) + C\int_0^T \|\nabla^{s+1}v(t)\|_{L^2}^2\,dt.
\]
\end{lemma}
\begin{proof}
For $k \in [2,m]$, we use \eqref{v_trans} to get
\[
	\begin{aligned}
		\int_0^T \| |\xi|^k \hat v(t) \|_{L^2} \,d t &\leq \int_0^T \| |\xi|^k e^{-|\xi|^2t} \hat v_0 \|_{L^2} \,d t + \int_0^T \left\| \int_0^t |\xi|^k e^{-|\xi|^2(t-\tau)} P(\xi) \rwhat{  (v \cdot \nabla)v}(\tau) \,d \tau \right\|_{L^2}  d t \\
		&\quad + \int_0^T \left\| \int_0^t |\xi|^k e^{-|\xi|^2(t-\tau)} P(\xi)\rwhat{ \rho(u-v)}(\tau) \,d \tau \right\|_{L^2}  d t\\
		&=: \sfK_1 + \sfK_2 + \sfK_3.
	\end{aligned}
\]
For $\sfK_1$,  when $t \in (0,1)$,
\begin{equation*}
	\begin{aligned}
		\| |\xi|^k e^{-|\xi|^2t} \hat v_0 \|_{L^2} \leq Ct^{-\frac {m-s}{2}} \| v_0 \|_{H^{k-(m-s)}} \leq Ct^{-\frac {m-s}{2}} \| v_0 \|_{H^{s}},
	\end{aligned}
\end{equation*}
and for $t \in [1,\infty)$
\[
	\begin{aligned}
		\| |\xi|^k e^{-|\xi|^2t} \hat v_0 \|_{L^2} \leq \| \hat  v_0 \|_{L^{\infty}} \| |\xi|^k e^{-|\xi|^2t} \|_{L^2} \leq Ct^{-\frac {k}{2} - \frac {d}{4}} \| v_0 \|_{L^1}.
	\end{aligned}
\]
Here, note that $m-s \in [1,2)$. So this implies
\[
\sfK_1 \le C\lt(\|v_0\|_{L^1} + \|v_0\|_{H^s}\rt).
\]
For $\sfK_2$, we deal with the case $k = m$ first. First, we employ
\[\begin{aligned}
\|v\|_{L^\infty} &\le C\left\{\begin{array}{lcl}\|\nabla^{s+1}v\|_{L^2}^{\frac{d/2-1}{s}}\|v\|_{L^{\frac{2d}{d-2}}}^{1-\frac{d/2-1}{s}} & \mbox{if} & d\ge3 \\ \|\nabla^2 v\|_{L^2}^{\frac 12} \|v\|_{L^2}^{\frac 12}& \mbox{if} & d=2
\end{array} \right.\\
&\le C\|\nabla v\|_{H^s} + C\e_1^{1/2}\|\nabla^2 v\|_{L^2}^{1/2},
\end{aligned}\]
\[
|x^r e^{-ax}| \le \lt(\frac{r}{a}\rt)^r e^{-r}, \quad \forall x>0, 
\]
and Proposition \ref{lem_moser} (vi)  to get
\[
	\begin{aligned}
		\left\|  |\xi|^m e^{-|\xi|^2(t-\tau)} P(\xi) \rwhat{ (v \cdot \nabla)v}  (\tau)\right\|_{L^2} & \leq C  (t-\tau)^{-\frac {m-s}{2}} \| \nabla^{s+1} (v \otimes v)(\tau) \|_{L^2} \\
		& \leq C (t-\tau)^{-\frac {m-s}{2}} \| \nabla^{s+1} v (\tau)\|_{L^2} \| v (\tau)\|_{L^{\infty}}  \\
		& \leq  C(t-\tau)^{-\frac {m-s}{2}} \| \nabla^{s+1} v(\tau) \|_{L^2} \lt(\|\nabla v(\tau)\|_{H^s} + \e_1^{1/2}\|\nabla^2 v(\tau)\|_{L^2}^{1/2} \rt)\\
		& \le C (t-\tau)^{-\frac {m-s}{2}} \| \nabla v(\tau) \|_{H^s}^2 + C\e_1 (t-\tau)^{-\frac {m-s}{2}} \| \nabla^2 v (\tau) \|_{H^{m-2}}.
	\end{aligned}
\]
Moreover, one has
\[
	\begin{aligned}
		\left\| |\xi|^m e^{-|\xi|^2(t-\tau)} P(\xi) \rwhat { (v \cdot \nabla)v}(\tau) \right\|_{L^2}& \le C  (t-\tau)^{-\frac {m}{2}} \| \nabla\cdot(v\otimes v)(\tau) \|_{L^2} \\
		& \leq C (t-\tau)^{-\frac {m}{2}} \| \nabla v (\tau)\|_{L^{\infty}} \| v(\tau) \|_{L^2},
	\end{aligned}
\]
where $C$ in the above two inequalities is independent of $t$, $\tau$ and $T$.

Thus, assuming $T>2$ without loss of generality, we use Minkowski's integral inequality to have
\begin{equation}\label{Hk_v_est_2}
\begin{aligned}
		&\int_0^T \left\| \int_0^t |\xi|^m e^{-|\xi|^2(t-\tau)} P(\xi) \rwhat{  (v \cdot \nabla)v}(\tau) \,d \tau \right\|_{L^2} \,d t \\
		&\quad \le  C\e_1 \int_0^2 \|\nabla^2 v(t)\|_{H^{m-2}}\int_0^t (t-\tau)^{-\frac{m-s}{2}}\,d\tau dt + C\int_0^2 \|\nabla^{s+1} v(t)\|_{L^2}^2 \int_0^t (t-\tau)^{-\frac{m-s}{2}}\,d\tau dt \\
		&\qquad +C \int_2^T \int_0^{t-1} (t-\tau)^{-\frac {m}{2}} \| \nabla v (\tau)\|_{L^{\infty}} \| v(\tau) \|_{L^2} \,d \tau + C\e_1 \int_2^T \int_{t-1}^t (t-\tau)^{-\frac {m-s}{2}} \| \nabla^2 v(\tau) \|_{H^{m-2}}\,d \tau\\
		&\qquad + C\int_2^T\int_{t-1}^t (t-\tau)^{-\frac{m-s}{2}}\|\nabla^{s+1} v(\tau)\|_{L^2}^2\,d\tau\\
		&\quad \le C\e_1 \int_0^T \|\nabla^2 v(t)\|_{H^{m-2}}\,dt + C\int_0^T \|\nabla^{s+1} v(t)\|_{L^2}^2\,dt.
\end{aligned}
\end{equation}
On the other hand, for $k = 2$, we use \eqref{intp_1} to get
\[
\left\| |\xi|^2 e^{-|\xi|^2(t-\tau)} P(\xi) \rwhat{  (v \cdot \nabla)v}(\tau) \right\|_{L^2} \le C (t-\tau)^{-\frac {3}{2}} \| \rwhat {(v \otimes v)}(\tau) \|_{L^2}   \leq C(t-\tau)^{-\frac {3}{2}} \| v(\tau) \|_{L^2}^{\frac {d+4}{d+2}} \| \nabla^2 v(\tau) \|_{H^{m-2}}^{\frac {d}{d+2}}.
\]
Moreover, we may get
\[
\left\| |\xi|^2 e^{-|\xi|^2(t-\tau)} P(\xi) \rwhat { (v \cdot \nabla)v}(\tau) \right\|_{L^2} \le C (t-\tau)^{-\frac 12} \| \rwhat {(v\cdot \nabla v)}(\tau) \|_{L^2}  \leq C(t-\tau)^{-\frac 12} \|\nabla v(\tau)\|_{L^\infty} \| v(\tau) \|_{L^2}.
\]
Thus, by assuming $T>2$, we obtain
\[\begin{aligned}
&\int_0^T \left\| \int_0^t |\xi|^2 e^{-|\xi|^2(t-\tau)} P(\xi) \rwhat{  (v \cdot \nabla)v}(\tau) \,d \tau \right\|_{L^2} \,d t \\
&\quad \leq C\e_1 \int_0^T \| \nabla^2 v(t) \|_{H^{m-2}} \,d t + C\int_0^T \|\nabla^{s+1}v(t)\|_{L^2}^2\,dt+ C \int_0^T \| v(t) \|_{L^2}^{\frac {d+4}{d+2}} \| \nabla^2 v(t) \|_{H^{m-2}}^{\frac {d}{d+2}} \,d t \\
&\quad \leq C \e_1 \int_0^T \| \nabla^2 v(t) \|_{H^{m-2}} \,d t + C\int_0^T \|\nabla^{s+1}v(t)\|_{L^2}^2\,dt +C \e_1\| v \|_{L^{\frac{d+4}{2}}(0,T;L^2)},
\end{aligned}\]
due to \eqref{est_vv}. Thus, for any $k \in [2,m]$, we have
\[\begin{aligned}
\sfK_2&\le C \e_1 \int_0^T \| \nabla^2 v(t) \|_{H^{m-2}} \,d t + C\int_0^T \|\nabla^{s+1}v(t)\|_{L^2}^2\,dt +C \e_1\| v \|_{L^{\frac{d+4}{2}}(0,T;L^2)}.
\end{aligned}\]
 
\noindent For $\sfK_3$, note that
\[
	\begin{aligned}
		\left\| |\xi|^k e^{-|\xi|^2(t-\tau)} P(\xi)\rwhat{ \rho(u-v)} (\tau) \right\|_{L^2} &\leq C (t-\tau)^{-\frac {k}{2}} \| e^{-|\xi|^2 \frac {t-\tau}{2}} \|_{L^2} \| \rwhat{ \rho (u-v)}(\tau) \|_{L^{\infty}}  \\
		&\leq C (t-\tau)^{-(\frac {k}{2}+\frac {d}{4})} \| \rho(\tau) \|_{L^2} \| (u-v)(\tau) \|_{L^2}.
	\end{aligned}
\]
Next, for $k \in [2,m]$, choose $\sigma>0$ satisfying $k-\sigma = m-s \in [1,2)$. Then
\[\begin{aligned}
&\left\| |\xi|^k e^{-|\xi|^2(t-\tau)} P(\xi)\rwhat {\rho(u-v)} (\tau) \right\|_{L^2} \\
&\quad \leq C (t-\tau)^{-\frac {k-\sigma}{2}} \| \rwhat{ \nabla^\sigma (\rho(u-v))}(\tau) \|_{L^2}  \\
&\quad \leq C (t-\tau)^{-\frac {m-s}{2}} (\| \nabla^\sigma \rho(\tau) \|_{L^2} \| (u-v)(\tau) \|_{L^\infty} + \| \hat \rho(\tau) \|_{L^p} \|\rwhat{  \nabla^\sigma(u-v)}(\tau) \|_{L^q})\\
&\quad \le C (t-\tau)^{-\frac {m-s}{2}} \|\rho(\tau)\|_{H^s}\|(u-v)(\tau)\|_{H^m},
\end{aligned}\]
where 
\[
\frac {1}{p} \in \lt(\frac {1}{2}, \frac {1}{2} + \frac {s}{d}\rt) \quad \mbox{and} \quad \frac {1}{q} \in \lt(\frac {1}{2},\frac {1}{2} + \frac {m-s}{d}\rt)\quad \mbox{with} \quad\frac {1}{p} + \frac {1}{q} = \frac {3}{2}, 
\]
and we used
\[
	\begin{aligned}
		\| \hat  \rho \|_{L^p} \leq C\| \rho \|_{H^s} \quad \mbox{and} \quad \| \rwhat{ \nabla^\sigma(u-v)} \|_{L^q} \leq C\|\nabla^\sigma (u-v) \|_{H^{m-s}} \le C\|(u-v)\|_{H^m}.
	\end{aligned}
\]
Thus, as \eqref{Hk_v_est_2}, we obtain
\[\begin{aligned}
&\int_0^T \left\| \int_0^t |\xi|^k e^{-|\xi|^2(t-\tau)} P(\xi)\rwhat{  \rho(u-v)}(\tau) \,d \tau \right\|_{L^2} d t \cr
&\quad  \leq C \sup_{t \in [0,T]} \| \rho(t) \|_{H^s} \int_0^T \| (u-v)(t) \|_{H^{m}} \,d t\\
&\quad \le C\e_1 \int_0^T \|(u-v)(t)\|_{L^2}\,dt + C\e_1 \int_0^T \|\nabla^2 u(t)\|_{H^{m-2}}\,dt + C\e_1 \int_0^T \|\nabla^2 v(t)\|_{H^{m-2}}\,dt,
\end{aligned}\]
and this implies, for any $k\in[2,m]$,
\[
\sfK_3\le C\e_1 \int_0^T \|(u-v)(t)\|_{L^2}\,dt + C\e_1 \int_0^T \|\nabla^2 u(t)\|_{H^{m-2}}\,dt+ C\e_1 \int_0^T \|\nabla^2 v(t)\|_{H^{m-2}}\,dt.
\]
We then gather all the estimates for $\sfK_i$'s and use the smallness of $\e_1$ to have
\[\begin{aligned}
&\int_0^T \|\nabla^2 v(t)\|_{H^{m-2}}\,dt \cr
& \quad \le C\lt( \sqrt{\mathfrak{X}_0(s)} + \|v_0\|_{L^1}\rt) +  C\e_1 \int_0^T \|\nabla^2 u(t)\|_{H^{m-2}}\,dt + C\int_0^T \|\nabla^{s+1}v(t)\|_{L^2}^2\,dt \\
&\qquad + C\e_1 \|v\|_{L^p(0,T;L^2)} + C\e_1\int_0^T \|(u-v)(t)\|_{L^2}\,dt\cr
&\quad \leq C\lt( \sqrt{\mathfrak{X}_0(s)} + \|v_0\|_{L^1}\rt) +  C\e_1 \mathfrak{D}(s;T) + C\int_0^T \|\nabla^{s+1}v(t)\|_{L^2}^2\,dt   + C\e_1 \|v\|_{L^{\frac{d+4}{2}}(0,T;L^2)}.
\end{aligned}\]
Finally, we combine the above with Lemma \ref{v_lp_est} to conclude the desired result.
\end{proof}

\begin{lemma}\label{uv_est}
For $T>0$, there exists $C>0$ independent of $T$ such that
\[
\int_0^T \|(u-v)(t)\|_{L^2}\,dt \leq C\sqrt{\mathfrak{X}_0(s)}  +  C\e_1 \mathfrak{D}(s;T) .
\]
\end{lemma}
\begin{proof}
We first have
\begin{equation*}
	\begin{aligned}
		\partial_t(u-v) + (u-v) = -(u \cdot \nabla)u - \mathbb{P}(v \cdot \nabla)v  - \Delta v- \mathbb{P}\rho(u-v).
	\end{aligned}
\end{equation*}
Thus, we have
\begin{equation*}
	\begin{aligned}
		\frac {d }{d  t} \| u-v \|_{L^2}+ \| u-v \|_{L^2} &\leq \| (u \cdot \nabla)u \|_{L^2} + \| (v \cdot \nabla)v \|_{L^2} + \| \rho(u-v) \|_{L^2} + \|\Delta v\|_{L^2} \\
		&\leq \| u \|_{L^2} \| \nabla u \|_{L^{\infty}} + \| v \|_{L^2} \| \nabla v \|_{L^{\infty}} + C\| \rho \|_{L^2} \| u-v \|_{H^m} + \|\Delta v\|_{L^2}.
	\end{aligned}
\end{equation*}
Integrating the both sides over $[0,T]$ gives
\[\begin{aligned}
		&\sup_{t \in [0,T]} \| (u-v)(t) \|_{L^2} + \int_0^T \| (u-v)(t) \|_{L^2} \,d  t \\
		&\quad \le  C\|u_0 -v_0\|_{L^2} +  \sup_{t \in [0,T]} \| u \|_{L^2} \int_0^T \| \nabla u(t) \|_{L^{\infty}} \,d  t + C \int_0^T \|\Delta v(t)\|_{L^2}\,dt\\
		&\qquad + \sup_{t \in [0,T]} \| v(t)\|_{L^2} \int_0^T \| \nabla v(t) \|_{L^{\infty}} \,d  t + C\sup_{t \in [0,T]} \| \rho(t) \|_{L^2} \int_0^T \| \nabla^2 (u-v)(t) \|_{H^{m-2}} \,d  t\\
		&\quad \le  C\|u_0 -v_0\|_{L^2} +  C\e_1 \int_0^T \| \nabla u(t) \|_{L^{\infty}} \,d  t  + C\e_1  \int_0^T \| \nabla v(t) \|_{L^{\infty}} \,d  t \\
		&\qquad + C\e_1 \int_0^T \| \nabla^2 u(t) \|_{H^{m-2}} \,d  t + C\int_0^T \|\nabla^2 v(t)\|_{H^{m-2}}\,dt.
\end{aligned}\]
This completes the proof.
\end{proof}

%
%
%
%
%

\subsection{Closing the a priori estimates}
Hence, we collect all the estimates from Lemmas \ref{rho_est_up}-\ref{v_lp_est} to yield the following uniform-in-$T$ estimate.

\begin{corollary}\label{unif_T_est}
For $T>0$, we can find a sufficiently small $\e_1>0$ satisfying
\[
\mathfrak{X}(s;T) \le C^*\lt( \mathfrak{X}_0(s)+ \|v_0\|_{L^1}^2 + \e_1 \lt(\sqrt{\mathfrak{X}_0(s)} + \|v_0\|_{L^1}\rt)\rt),
\]
where $C^*>0$ is independent of $T$.
\end{corollary}
\begin{proof}
First, we combine Lemmas \ref{v_lp_est}-\ref{uv_est} to yield
\begin{align*}
\mathfrak{D}(s;T) &\le C\lt( \sqrt{\mathfrak{X}_0(s)} + \|v_0\|_{L^1}\rt) + C\e_1 \mathfrak{D}(s;T) + C\int_0^T \|\nabla v(t)\|_{H^s}^2\,dt,
\end{align*}
or
\[
\mathfrak{D}(s;T) \le C\lt( \sqrt{\mathfrak{X}_0(s)} + \|v_0\|_{L^1}\rt)  + C\int_0^T \|\nabla v(t)\|_{H^s}^2\,dt.
\]
We also deduce from Lemma \ref{v_Hs_est} to yield
\[
\int_0^T \|\nabla v(t)\|_{H^s}^2\,dt \le C\mathfrak{X}_0(s) + C\e_1^2 \mathfrak{D}(s;T).
\]
These two relations imply
\[
\mathfrak{D}(s;T) \le C\lt(\sqrt{\mathfrak{X}_0(s)} + \|v_0\|_{L^1} \rt) + C\mathfrak{X}_0(s),
\]
for sufficiently small $\e_1>0$. Hence, from Lemma \ref{rho_est_up}, we have
\[
\sup_{0 \le t \le T} \|\rho(t)\|_{H^s}^2 \le \|\rho_0\|_{H^s}^2 \exp\lt(\mathfrak{D}(s;T) \rt) \le C\mathfrak{X}_0(s).
\]
We also obtain from Lemmas \ref{u_Hm_est}-\ref{v_Hs_est} that 
\[\begin{aligned}
&\sup_{0 \le t \le T} \lt(\|u(t)\|_{H^m}^2 + \|v(t)\|_{H^s}^2 \rt) + \int_0^T \|\nabla v(t)\|_{H^s}^2\,dt \\
&\quad \le C\mathfrak{X}_0(s) + C\e_1 \mathfrak{D}(s;T)\\
&\quad \le C\mathfrak{X}_0(s) + C\e_1\lt( \sqrt{\mathfrak{X}_0(s)} + \|v_0\|_{L^1}\rt) + C\e_1  \int_0^T \|\nabla v(t)\|_{H^s}^2\,dt,
\end{aligned}\]
and this gives
\[
\sup_{0 \le t \le T} \lt(\|u(t)\|_{H^m}^2 + \|v(t)\|_{H^s}^2 \rt) \le  C\mathfrak{X}_0(s) + C\e_1\lt( \sqrt{\mathfrak{X}_0(s)} + \|v_0\|_{L^1}\rt).
\]
Finally, we already have
\[
\|v\|_{L^\frac{d+4}{2}(0,T;L^2)}^2 \le \mathfrak{D}(s;T)^2 \le C\lt( \mathfrak{X}_0(s) + \|v_0\|_{L^1}^2\rt),
\]
where we used the smallness of the initial data, and this completes the assertion.
\end{proof}

\subsection{Proof of Theorem \ref{main_thm1}: existence} Now we present the proof for Theorem \ref{main_thm1}. We choose a positive constant $\e_1\ll1$ sufficiently small so that it satisfies the required smallness condition in Lemmas \ref{u_Hm_est}--\ref{v_lp_est}. Then, assume that
\[
\sqrt{\mathfrak{X}_0(s)} + \|v_0\|_{L^1} \le\delta \e_1,
\]
where $\delta>0$ satisfies
\[
C^*(\delta^2 + \delta) <1,
\]
and $C^* > 0$ appeared in Corollary \ref{unif_T_est}. Then we set
\[
\mathcal{S} := \lt\{ T\ge 0 \ | \ \mathfrak{X}(s;T) <\e_1^2\rt\}. 
\]
By the local-in-time existence theorem in Theorem \ref{thm_local}, the set $\mathcal{S}$ is non-empty. Now, we argue by contradiction to show $\sup\mathcal{S} = \infty$. Assume that $T^*:= \sup\mathcal{S} <\infty$. Then we have
\[\begin{aligned}
\e_1^2 = \lim_{t \to T*-}\mathfrak{X}(s;t) &\le C^*\lt(\mathfrak{X}_0(s) + \|v_0\|_{L^1}^2 + \e_1 \lt(\sqrt{\mathfrak{X}_0(s)} + \|v_0\|_{L^1}\rt)\rt)\\
&\le C^* \lt( \lt(\sqrt{\mathfrak{X}_0(s)} + \|v_0\|_{L^1}\rt)^2 + \e_1 \lt(\sqrt{\mathfrak{X}_0(s)} + \|v_0\|_{L^1}\rt)\rt)\\
&\le C^* (\delta^2 + \delta)\e_1^2 <\e_1^2,
\end{aligned}\]
which leads to a contradiction. This implies $T^* = \infty$, and hence the classical solution obtained in Theorem \ref{thm_local} globally exists. Moreover, according to the arguments in Corollary \ref{unif_T_est}, we also have
\[
\| \nabla^2 u \|_{L^{1}(0,\infty;H^{m-2})} + \| \nabla^2 v \|_{L^{1}(0,\infty;H^{m-2})} + \| \nabla u \|_{L^{1}(0,\infty;L^{\infty})} + \| \nabla v \|_{L^{1}(0,\infty;L^{\infty})} + \| \nabla v \|_{L^2(0,\infty;H^{s})} <\infty.
\]
This completes the proof.

%
%
%
%
%

\subsection{Proof of Theorem \ref{main_thm1}: stability}\label{ssec_sta} In this subsection, we investigate the stability estimates for system \eqref{A-1}. It is worth noticing that this stability estimate can be used to show the local existence and uniqueness of solutions to \eqref{A-1} based on the Cauchy sequence construction argument. 

Note that we deal with solutions $\rho$ of rather weak regularity, precisely, $\rho(t) \in H^s$ where $s$ can be $s = \frac d2 - 1 + \epsilon$ for some  $\epsilon \ll 1$. Thus, in particular, in the two-dimensional case, $\rho(t) \in H^\e$ for any $\epsilon > 0$. Thus, the stability estimate for $\rho$ in $L^2$ is not expected to be applicable when $s$ is close to $\frac d2 -1$. For that reason, we employ the negative Sobolev space for the stability estimate for $\rho$. To be more specific, for $\alpha \in (0,1/2)$ satisfying $\frac d2 -2\alpha <s$, we consider the $\dot{H}^{-\alpha}$-norm for the stability estimate for $\rho$. We notice that if $\rho \in L^1 \cap H^s$, then $\rho \in \dot{H}^{-\alpha}$. Indeed, we find
\[\begin{aligned}
\|\rho\|_{\dot{H}^{-\alpha}}^2 &= \intr \rho K * \rho\,dx\\
&\le C \intr \rho(x) \lt( \lt(\int_{\{ |x-y| \le 1\}} + \int_{\{ |x-y|>1\}}\rt) \frac{1}{|x-y|^{d-2\alpha}}\rho(y)\,dy \rt) dx\\
&\le C \intr \rho(x) \lt(\|\rho\|_{L^q} + \|\rho\|_{L^1} \rt) dx\\
&\le C\|\rho\|_{L^1}\lt(\|\rho\|_{H^s} + \|\rho\|_{L^1} \rt),
\end{aligned}\]
where 
\[
\frac 1q < \frac12 - \frac{d/2-2\alpha}{d}
\] 
and $K$ is the Riesz potential, i.e., $K = c_{d,\alpha}|x|^{-(d-2\alpha)}$ for some $c_{d,\alpha} > 0$ depending only on $\alpha$ and $d$, which can be regarded as $\Lambda^{-2\alpha}$. As expected, if $s$ is large enough, then we can consider the $L^2$-norm for the stability estimate for $\rho$, see Remark \ref{rmk_l2_sta} below.
 
\begin{lemma}\label{stab_1}
For each $i=1,2$, if $(\rho_i, u_i, v_i)$  is a solution to \eqref{A-1} corresponding to the initial data $(\rho_i^0, u_i^0, v_i^0)$, then we have
\[\begin{aligned}
&\|(\rho_1 - \rho_2)(t)\|_{\dot{H}^{-\alpha}}^2 + \|(u_1 -u_2)(t)\|_{H^1}^2 + \|(v_1 -v_2)(t)\|_{L^2}^2 \\
&\quad\le \lt(\|\rho_1^0 - \rho_2^0\|_{\dot{H}^{-\alpha}}^2 + \|u_1^0 -u_2^0\|_{H^1}^2 + \|v_1^0 -v_2^0\|_{L^2}^2 \rt)  \exp\lt(C(1+t) \rt),
\end{aligned}\]
where $\alpha \in (0,1)$ satisfies $\frac d2 -\alpha <s$ and $C>0$ is a positive constant.
\end{lemma}
\begin{proof}
We use \cite[Theorem 1.1]{CJpre2} to get
\[\begin{aligned}
\frac12\frac{d}{dt}\|\rho_1 - \rho_2\|_{\dot{H}^{-\alpha}}^2 &= \frac12\frac{d}{dt}\intr (\rho_1 - \rho_2) K * (\rho_1 - \rho_2)\,dx\\
&\le \intr \rho_1(u_1-u_2) \cdot \nabla K * (\rho_1 -\rho_2) \,dx + C\|\nabla u_2\|_{L^\infty} \intr (\rho_1-\rho_2) K * (\rho_1 -\rho_2)\,dx\\
&\le \|\Lambda^{1-\alpha} (\rho_1(u_1-u_2))\|_{L^2} \|\Lambda^{-1+\alpha}\nabla K * (\rho_1 - \rho_2)\|_{L^2}+ C\|\rho_1-\rho_2\|_{\dot{H}^{-\alpha}}^2\\
&\le C \lt(\|\Lambda^{1-\alpha}\rho_1\|_{L^{p_1}}\|u_1-u_2\|_{L^{q_1}} + \|\rho_1\|_{L^{p_2}}\|\Lambda^{1-\alpha}(u_1-u_2)\|_{L^{q_2}}\rt)\|\rho_1 - \rho_2\|_{\dot{H}^{-\alpha}}\\
&\quad + C\|\rho_1-\rho_2\|_{\dot{H}^{-\alpha}}^2\\
&\le C\lt(\|u_1- u_2\|_{H^1}^2 + \|\rho_1 -\rho_2\|_{\dot{H}^{-\alpha}}^2 \rt),
\end{aligned}\]
where $p_i$ and $q_i$ are given by
\bq\label{pq_choice}
\frac{1}{p_i}  = \frac12 - \frac{\theta_i}{d}, \quad \frac{1}{q_i}  = \frac12 - \frac{d/2-\theta_i}{d}, \quad i=1,2,
\eq
and $\theta_i$, $i=1,2$ satisfy
\[
0< (1-\alpha) + \theta_1 \le s, \quad \frac d2-1 < \theta_1 < \frac d2 \quad \mbox{so that} \quad d/2-\theta_1 \in (0,1)
\]
and
\[
0< \theta_2 \le s, \quad \frac d2-1 < 1-\alpha + \theta_2 < \frac d2 \quad \mbox{so that} \quad d/2-\theta_2 \in (0,1).
\]
Here, straightforward computation yields
\[
\frac{d}{dt}\|u_1 - u_2\|_{H^1}^2 \le C\lt(\|u_1 - u_2\|_{H^1}^2 + \|v_1 - v_2\|_{L^2}^2\rt) + \frac 14\|\nabla (v_1 - v_2)\|_{L^2}^2.
\]
Next, we estimate
\begin{align*}
\frac12&\frac{d}{dt}\|v_1-v_2\|_{L^2}^2 + \|\nabla (v_1 -v_2)\|_{L^2}^2 \\
&= \intr (v_1 -v_2) \cdot ((v_1 - v_2) \cdot \nabla v_1)\,dx + \intr (v_1 -v_2)\cdot (\rho_1 -\rho_2)(u_1 -v_1)\,dx\\
&\quad + \intr (v_1 -v_2)\cdot \rho_2 (u_1-u_2 -(v_1- v_2))\,dx\\
&\le C\|\nabla v_1\|_{L^\infty}\|v_1 - v_2\|_{L^2}^2 + \|\rho_1 -\rho_2\|_{\dot{H}^{-\alpha}}\|(v_1-v_2)\cdot(u_1-v_1)\|_{\dot{H}^\alpha}  + \|v_1 -v_2\|_{L^2}\|\rho_2\|_{L^{p_1}}\|u_1 -u_2\|_{L^{q_1}}\\
&\le C\|\nabla v_1\|_{L^\infty}\|v_1 -v_2\|_{L^2}^2 +C(1+ \|v_1\|_{L^\infty}) \|\rho_1 - \rho_2\|_{\dot{H}^{-\alpha}}\|\nabla (v_1 -v_2)\|_{L^2}\\
&\quad + C(1+\|v_1\|_{L^\infty}+\|\nabla v_1\|_{L^\infty})\|\rho_1 - \rho_2\|_{\dot{H}^{-\alpha}}\|(v_1 -v_2)\|_{L^2} + C\|v_1-v_2\|_{L^2}\|u_1 - u_2\|_{H^1}\\
&\le C(1+\|\nabla v_1\|_{L^\infty})\lt( \|v_1 -v_2\|_{L^2}^2 + \|\rho_1 - \rho_2\|_{\dot{H}^{-1}}^2 +  \|u_1 - u_2\|_{H^1}^2\rt) \\
&\quad +C(1+ \|v_1\|_{H^{s+1}}) \|\rho_1 - \rho_2\|_{\dot{H}^{-\alpha}}\| v_1 -v_2\|_{H^1}\\
&\le  C(1+\|\nabla v_1\|_{L^\infty}+ \|\nabla^{s+1} v_1\|_{L^2}^2)\lt( \|v_1 -v_2\|_{L^2}^2 + \|\rho_1 - \rho_2\|_{\dot{H}^{-\alpha}}^2 +  \|u_1 - u_2\|_{H^1}^2\rt)  +\frac14\|\nabla (v_1 -v_2)\|_{L^2}^2,
\end{align*}
where we used $s+1 > \frac d2$ and $p_1$ and $q_1$ are from \eqref{pq_choice}.

Thus, we combine all the estimates to yield
\[\begin{aligned}
\frac{d}{dt}&\lt(\|\rho_1 - \rho_2\|_{\dot{H}^{-\alpha}}^2+ \|v_1 -v_2\|_{L^2}^2 +  \|u_1 - u_2\|_{H^1}^2\rt) \\
&\le C\lt(1+ \|\nabla v_1\|_{L^\infty} + \|\nabla^{s+1} v_1\|_{L^2}^2 \rt)\lt(\|\rho_1 - \rho_2\|_{\dot{H}^{-\alpha}}^2+ \|v_1 -v_2\|_{L^2}^2 +  \|u_1 - u_2\|_{H^1}^2\rt) ,
\end{aligned}\]
and use the Gr\"onwall's lemma to conclude the proof.
\end{proof}

\begin{remark}\label{rmk_l2_sta}
When $s> \frac d2$, we may consider $L^2$-stability estimate of $\rho$ since
\begin{align*}
\frac12\frac{d}{dt}\|\rho_1 -\rho_2\|_{L^2}^2 &= -\intr (\rho_1 - \rho_2) \lt( \nabla(\rho_1 -\rho_2) \cdot u_1 + (\rho_1 -\rho_2)\nabla \cdot u_1 \rt)\,dx\\
&\quad -\intr (\rho_1- \rho_2) (\nabla \rho_2 \cdot (u_1 -u_2) + \rho_2 \nabla \cdot (u_1 -u_2))\,dx\\
&\le C\|\nabla u_1\|_{L^\infty}\|\rho_1 - \rho_2\|_{L^2}^2  + C\|\rho_1 - \rho_2\|_{L^2}\|\nabla \rho_2\|_{L^p}\|u_1 -u_2\|_{L^q} \\
&\quad + C\|\rho_1 -\rho_2\|_{L^2}\|\rho_2\|_{L^\infty}\|\nabla (u_1 - u_2)\|_{L^2}\\
&\le C(\|\rho_1 - \rho_2\|_{L^2}^2 + \|u_1- u_2\|_{H^1}^2),
\end{align*}
where $p$ and $q$ are given by
\[
\frac1p  = \frac12 - \frac{\theta-1}{d}, \quad \frac1q = \frac{\theta-1}{d} = \frac12 - \lt(\frac12- \frac{\theta-1}{d}\rt) = \frac12 - \frac{d-2(\theta-1)}{2d},
\]
and $\theta$ satisfies
\[
0<\theta \le s, \quad \frac d2 < \theta < \frac d2+1, \quad \mbox{so that} \quad \frac{d-2(\theta-1)}{2} \in (0,1).
\]
In this case, we may estimate the $L^2$-stability of $v$ as
\begin{align*}
&\frac12\frac{d}{dt}\|v_1-v_2\|_{L^2}^2 + \|\nabla (v_1 -v_2)\|_{L^2}^2 \cr
&\quad \le C\|\nabla v_1\|_{L^\infty}\|v_1 - v_2\|_{L^2}^2 + C(1+\|v_1\|_{L^\infty}) \|\rho_1 -\rho_2\|_{L^2}\|(v_1-v_2)\|_{L^2}  + \|v_1 -v_2\|_{L^2}\|\rho_2\|_{L^\infty}\|u_1 -u_2\|_{L^2}\\
&\quad \le  C(1+\|\nabla v_1\|_{L^\infty}+ \|\nabla^{s+1} v_1\|_{L^2}^2)\lt( \|v_1 -v_2\|_{L^2}^2 + \|\rho_1 - \rho_2\|_{L^2}^2 +  \|u_1 - u_2\|_{L^2}^2\rt),
\end{align*}
and again use the Gr\"onwall's lemma to get the desired $L^2$-stability:
\[\begin{aligned}
&\|(\rho_1 - \rho_2)(t)\|_{L^2}^2 + \|(u_1 -u_2)(t)\|_{H^1}^2 + \|(v_1 -v_2)(t)\|_{L^2}^2  \cr
&\quad  \le \lt(\|\rho_1^0 - \rho_2^0\|_{L^2}^2 + \|u_1^0 -u_2^0\|_{H^1}^2 + \|v_1^0 -v_2^0\|_{L^2}^2 \rt)  \exp\lt(C(1+t) \rt).
\end{aligned}\]
\end{remark}

%
%
%
%
%
%

\section{Large time behavior estimates}\label{sec:4}
\setcounter{equation}{0}
In this section, we investigate the large time behavior of solutions, constructed in Theorem \ref{main_thm1}, to the system \eqref{A-1}. For this, we set
\[\mm_f^k (t) := \sup_{\tau \in [0,t]} \tau^{\frac{d}{4} + \frac k2}  \| \nabla^k  f(\tau)\|_{L^2}, \quad \mm_f(t) := \sup_{k \in [0,\, \min\{m, \frac d2 + 2\}]}\mm_f^k(t), \quad f = u \  \mbox{ or } \ v,\]
\[
\md^\theta(t) := \sup_{\tau \in [0,t] } \tau^{\frac{d}{4} + \frac{\theta}{2}+1} \|\nabla^\theta(u-v)(\tau)\|_{L^2}, \quad \mbox{and} \quad \md(t) :=\sup_{\theta \in [0, \,\min\{m-2, \frac d2\} ]}\md^\theta(t).
\]
In addition, in the case $m \ge \frac d2+2$, we set 
\[
\tilde{\mm}_v (t) := \sup_{\tau \in [0,t]} \tau^{\frac{d}{2} +1}  \| |\xi|^2 \hat v (\tau)\|_{L^1}
\quad \mbox{and} \quad 
\tilde{\md}(t) := \sup_{\tau \in [0,t] } \tau^{\frac{d}{2} + 1} \|\rwhat{(u-v)}(\tau)\|_{L^1}.
\]
Here, we assume $\|\Delta u_0\|_{L^\infty}<\infty$ when $m=\frac d2 +2$.

Then the following is direct.
\begin{lemma}\label{u_vel}
For $t>0$ and $k \in [0,\,\min\{m, \frac d2 + 2\}]$, there exists $C>0$ independent of $t$ such that 
\[
 \mm_u^k (t) \le e^{-t}\|\nabla^k u_0\|_{L^2}+ C\mm_v^k (t).
\]
\end{lemma}

In the following two lemmas, we provide the upper bound estimates of $\mm_v(t)$ and $\md(t)$. 

\begin{lemma}\label{v_vel2}
For $t>0$ and $k \in [0, \,\min\{m, \frac d2 + 2\}]$, we have
\[\begin{aligned}
t^{\frac d4 + \frac k2} \|\nabla^k v(t)\|_{L^2} \le C\|v_0\|_{L^1} + C\e_1^2  + C\mm_v^0(t) \mm_v(t) + C\e_1 \lt( \tilde{\md}(t){\bf 1}_{ \{ m \geq \frac d2 + 2 \}} + \md(t) + \mm_u(t) + \mm_v(t)\rt). 
\end{aligned}\]
In particular, this implies
\[\begin{aligned}
\mm_v(t) \le C\|v_0\|_{L^1} + C\e_1^2  + C\mm_v^0(t) \mm_v(t) + C\e_1 \lt( \tilde{\md}(t){\bf 1}_{ \{ m \geq \frac d2 + 2 \}} + \md(t) + \mm_u(t)\rt). 
\end{aligned}\]
\end{lemma}
\begin{proof}
We split the proof into two cases; $m < \frac d2 +2$ and $m \ge \frac d2 +2$. The proof is rather lengthy and technical, thus for smoothness of reading we postpone the proof of the latter case to Appendix \ref{app_lem1} and provide the details of proof in the case of $m < \frac d2 +2$ here. 

From \eqref{v_trans}, we separately estimate as
\[\begin{aligned}
\| |\xi|^k \hat v(t)\|_{L^2} &\le \| |\xi|^k e^{-|\xi|^2 t }\hat v_0\|_{L^2} + \int_0^t \| |\xi|^k e^{-|\xi|^2 (t-\tau)} P(\xi) \rwhat{v \cdot \nabla v}(\tau)\|_{L^2}\,d\tau\\
&\quad +  \int_0^t \| |\xi|^k e^{-|\xi|^2 (t-\tau)} P(\xi) \rwhat{\rho(u-v)}(\tau)\|_{L^2}\,d\tau\\
& =: \sfI_1 + \sfI_2 + \sfI_3.
\end{aligned}\]
Here we first easily estimate $\sfI_1$ as
\[
\sfI_1 \le Ct^{-\frac{d}{4}-\frac k2} \|v_0\|_{L^1}.
\]
For $\sfI_2$,  we use Lemma \ref{lem_moser} (vi) and (vii) to get
\[\begin{aligned}
\sfI_2 &\le C\int_0^{t/2} (t-\tau)^{-\frac{d}{4}-\frac k2-\frac12} \|(v \otimes v)(\tau)\|_{L^1}\,d\tau+ C\int_{t/2}^t (t-\tau)^{-\frac12} \| |\xi|^k \rwhat{v \otimes v}(\tau)\|_{L^2}\,d\tau\\
&\le C\e_1^{2-\frac 2d}t^{-\frac d4 -\frac k2-\frac12}  [\mm_v^0(t)]^{\frac 2d}\int_0^{t/2} \tau^{-\frac 12}\,d\tau + C\int_{t/2}^t (t-\tau)^{-\frac 12} \|\nabla^k v(\tau)\|_{L^2}\|\hat v(\tau)\|_{L^1}\,d\tau\\
&\le C \e_1^{2-\frac 2d}t^{-\frac d4 -\frac k2} [\mm_v^0(t)]^{\frac 2d} + Ct^{-\frac d4 -\frac k2} \mm_v^k(t) \int_{t/2}^t (t-\tau)^{-\frac12}\|\nabla^{s+1}v(\tau)\|_{L^2}^{\frac{d/2}{s+1}}\|v(\tau)\|_{L^2}^{1-\frac{d/2}{s+1}}\,d\tau\\
&\le C \e_1^{2-\frac 2d}t^{-\frac d4 -\frac k2} [\mm_v^0(t)]^{\frac 2d} \\
&\quad + Ct^{-\frac d4 -\frac k2} \mm_v^k(t)\lt( \int_{t/2}^t \|\nabla^{s+1}v(\tau)\|_{L^2}^{2}\,d\tau\rt)^{\frac{d/2}{2(s+1)}}\lt(\int_{t/2}^t (t-\tau)^{-\frac{s+1}{2(s+1)-d/2}}\|v(\tau)\|_{L^2}^{\frac{2(s+1-d/2)}{2(s+1)-d/2}}\,d\tau \rt)^{1-\frac{d/2}{2(s+1)}}\\
&\le C\e_1^{2-\frac 2d} t^{-\frac d4 -\frac k2}  [\mm_v^0(t)]^{\frac 2d} \\
&\quad + Ct^{-\frac d4 -\frac k2} \mm_v^k(t)\e_1^{\frac{d/2}{(s+1)}}\lt(\e_1^{\frac{s+1-d/2}{2(s+1)-d/2}}[\mm_v^0(t)]^{\frac{s+1-d/2}{2(s+1)-d/2}}\int_{t/2}^t (t-\tau)^{-\frac{s+1}{2(s+1)-d/2}}\tau^{-\frac{(s+1-d/2)}{2(s+1)-d/2}}\,d\tau \rt)^{1-\frac{d/2}{2(s+1)}}\\
&\le C\e_1^{2-\frac 2d}t^{-\frac d4 -\frac k2}  [\mm_v^0(t)]^{\frac 2d}+ C \e_1^{\frac{s+1+d/2}{2(s+1)}}t^{-\frac d4 -\frac k2} \mm_v^k(t)[\mm_v^0(t)]^{\frac{s+1-d/2}{2(s+1)}},
\end{aligned}\]
where we used the following result in Lemma \ref{v_Hs_est}:
\[
\int_0^t \|\nabla v (\tau)\|_{H^s}^2\,d\tau \le C\e_1^2,
\]
and we used the fact $s+1>\frac d2$ to get
\bq\label{v_inf_temp}
\|\hat v\|_{L^1} \le C \||\xi|^{s+1} \hat v\|_{L^2}^{\frac{d/2}{s+1}}\|\hat v\|_{L^2}^{1-\frac{d/2}{s+1}}
\eq
and
\[
 \int_{t/2}^t (t-\tau)^{-\frac{s+1}{2(s+1)-d/2}}\tau^{-\frac{(s+1-d/2)}{2(s+1)-d/2}}\,d\tau \le {\rm B}\lt(\frac{s+1}{2(s+1)-d/2}, \frac{s+1-d/2}{2(s+1)-d/2}\rt).
\]
Here ${\rm B}(\cdot, \cdot)$ denotes the beta function.\\

\noindent For $\sfI_3$, we separately estimate it as
\[\begin{aligned}
\sfI_3 &\le C\int_0^{t/2} (t-\tau)^{-\frac{d}{4}-\frac k2} \|(\rho(u-v))(\tau)\|_{L^1}\,d\tau + C\int_{t/2}^t \lt\| \mathds{1}_{\{ |\xi|\le 1\}}|\xi|^k e^{-|\xi|^2 (t-\tau)} P(\xi) \rwhat{\rho(u-v)}(\tau) \rt\|_{L^2}\,d\tau\\
&\quad + C\int_{t/2}^t \lt\| \mathds{1}_{\{ |\xi|>1\}}|\xi|^k e^{-|\xi|^2 (t-\tau)} P(\xi) \rwhat{\rho(u-v)}(\tau) \rt\|_{L^2}\,d\tau\\
&=: \sfI_{31} + \sfI_{32} + \sfI_{33}.
\end{aligned}\]
Here we estimate the terms $\sfI_{3i},i=1,2,3$ by dividing them into two cases: Case (i) $t>1$ and Case (ii) $t\le 1$.

\noindent $\diamond$ (Case (i): estimate of $\sfI_{31}$) Since we assume $t>1$, we have
\[\begin{aligned}
\sfI_{31} &\le Ct^{-\frac d4 -\frac k2} \lt(\int_0^{1/4} + \int_{1/4}^{t/2}\rt) \|\rho(\tau)\|_{L^2}\|(u-v)(\tau)\|_{L^2}\,d\tau\\
&\le C \e_1^2 t^{-\frac d4 - \frac k2} + C\e_1t^{-\frac d4 - \frac k2 } \md^0(t).
\end{aligned}\]

\noindent $\diamond$ (Case (i): estimate of $\sfI_{32}$) Here, when $k<2$, we choose $\eta_1>0$ satisfying $\eta_1<\min\{\frac d2, 2-k\}$ so that 
\[\begin{aligned}
\sfI_{32}&\le C\int_{t/2}^t (t-\tau)^{-\frac k2- \frac{\eta_1}{2}} \| |\xi|^{- \eta_1}\mathds{1}_{\{ |\xi|\le 1\}} \rwhat{\rho(u-v)}\|_{L^2}\,d\tau\\
&\le C\int_{t/2}^t (t-\tau)^{-\frac k2 -\frac{\eta_1}{2}} \||\xi|^{-\eta_1}\mathds{1}_{\{|\xi|\le1\}}\|_{L^2} \|\rho(\tau)\|_{L^2}\|(u-v)(\tau)\|_{L^2}\,d\tau\\
&\le C\e_1 t^{-\frac d4 -1}\md^0(t)\int_{t/2}^t (t-\tau)^{-\frac k2-\frac{\eta_1}{2}}\,d\tau\\
&\le C\e_1 t^{-\frac d4 -\frac k2}\md^0(t).
\end{aligned}\]
When $k= 2$, we choose $\eta_2>0$ satisfying $\eta_2<\frac d2$ and use $t>1$ so that
\[\begin{aligned}
\sfI_{32}&\le C\int_{t/2}^{t-1/4} (t-\tau)^{-1 - \frac{\eta_1}{2}} \| |\xi|^{-\eta_2}\mathds{1}_{\{ |\xi|\le 1\}} \rwhat{\rho(u-v)}(\tau)\|_{L^2}\,d\tau\\
&\quad + C\int_{t-1/4}^t \||\xi|^{-\eta_1} \mathds{1}_{\{ |\xi|\le 1\}}\rwhat{\rho(u-v)}(\tau)\|_{L^2}\,d\tau \\
&\le C \int_{t/2}^{t-1/4}  (t-\tau)^{-1 - \frac{\eta_2}{2}}  \||\xi|^{-\eta_1}\mathds{1}_{\{|\xi|\le1\}}\|_{L^2} \|\rho(\tau)\|_{L^2}\|(u-v)(\tau)\|_{L^2} \,d\tau\\
&\quad + C\int_{t-1/4}^{t} \||\xi|^{-\eta_1}\mathds{1}_{\{|\xi|\le1\}}\|_{L^2} \|\rho(\tau)\|_{L^2}\|(u-v)(\tau)\|_{L^2}\,d\tau\\
&\le C\e_1 t^{-\frac{d}{4}-1}\md^0(t) \int_{t/2}^{t-1/4} (t-\tau)^{-1 -\frac{\eta_2}{2}}\,d\tau + C\e_1 t^{-\frac{d}{4}-1}\md^0(t)\\
&\le C\e_1 t^{-\frac{d}{4}-1}\md^0(t).
\end{aligned}\]
When $k>2$, we use $0<k\le m<s+2$ and (vi) of Lemma \ref{lem_moser} to yield
\[\begin{aligned}
\sfI_{32}&\le C\int_{t/2}^{t-1/4} (t-\tau)^{-\frac k2} \| \rwhat{\rho(u-v)}(\tau)\|_{L^2}\,d\tau + C\int_{t-1/4}^t \|\rwhat{\rho(u-v)}(\tau)\|_{L^2}\,d\tau \\
&\le C\int_{t/2}^{t-1/4} (t-\tau)^{-\frac k2} \|(u-v)(\tau)\|_{L^p}\|\rho(\tau)\|_{L^q}\,d\tau + C\int_{t-1/4}^t \|(u-v)(\tau)\|_{L^p}\|\rho(\tau)\|_{L^q}\,d\tau \\
&\le C \int_{t/2}^{t-1/4}  (t-\tau)^{-\frac k2} \|\rho(\tau)\|_{H^s}\lt(\|\nabla^{k-2}(u-v)(\tau)\|_{L^2} + \|\nabla^k u(\tau)\|_{L^2} + \|\nabla^k v(\tau)\|_{L^2} \rt) d\tau \\
&\quad +  C\int_{t-1/4}^{t} \|\rho(\tau)\|_{H^s}\lt(\|\nabla^{k-2}(u-v)(\tau)\|_{L^2} + \|\nabla^k u(\tau)\|_{L^2} + \|\nabla^k v(\tau)\|_{L^2} \rt) d\tau\\
&\le C\e_1 t^{-\frac{d}{4}-\frac k2} \lt( \md^{k-2}(t) + \mm_u^k(t) + \mm_v^k(t)\rt),
\end{aligned}\]
where $p$ and $q$ satisfies
\[
\frac1p = \lt\{\begin{array}{lll} \displaystyle \frac12 - \frac{k-2+\eta_2}{d} & \mbox{if} & k < \min\{3, m\}\\[4mm]
 \displaystyle \frac12 - \frac{k-2}{d}&\mbox{if}& 3 \le k \le m\end{array}\rt., \quad \frac1q = \frac 12 - \frac1p  > \max\lt\{0,  \frac12 - \frac{s}{d}\rt\},
\]
and $\eta_2\in(0,1)$ satisfies $0<d/2-(k-2)-\eta_2 \le s$ when $k <3$. Thus, we have
\[
\sfI_{32} \le C\e_1 t^{-\frac d4 - \frac k2}\lt\{ \begin{array}{lll} \md^0(t) &\mbox{if} & k \leq 2,\\[2mm]
 \md^{k-2}(t) + \mm_u^k(t) + \mm_v^k(t) &\mbox{if}& k > 2. \end{array}\rt.
\]

\noindent $\diamond$ (Case (i): estimate of $\sfI_{33}$)
For $\sfI_{33}$, we use $(1+|\xi|^2)/2 \le |\xi|^2$ when $|\xi|>1$ and again estimate $\sfI_{33}$ depending on $k$. Suppose $k < 2$. Then we choose $\eta_3\in (0,1)$ satisfying $d/2-\eta_3 < s$ so that 
\[\begin{aligned}
\sfI_{33}&\le C \int_{t/2}^t \lt\||\xi|^k \mathds{1}_{\{|\xi|>1 \}}e^{-\frac{(1+|\xi|^2)}{2}(t-\tau)} \rwhat{\rho (u-v)}(\tau) \rt\|_{L^2} d\tau\\
&\le C\int_{t/2}^t e^{-\frac{t-\tau}{2}}(t-\tau)^{-\frac k2}  \|(\rho(u-v))(\tau)\|_{L^2}\,d\tau\\
&\le C\int_{t/2}^t e^{-\frac{t-\tau}{2}}(t-\tau)^{-\frac k2} \|\rho(\tau)\|_{L^p}\|(u-v)(\tau)\|_{L^q}\,d\tau\\
&\le C\int_{t/2}^t e^{-\frac{t-\tau}{2}}(t-\tau)^{-\frac k2} \|\rho(\tau)\|_{H^s}\lt( \|(u-v)(\tau)\|_{L^2} + \|\Delta u(\tau)\|_{L^2} + \|\Delta v(\tau)\|_{L^2}\rt) d\tau\\
&\le C\e_1 t^{-\frac d4 -\frac k2}(\md^0(t) + \mm_u^2(t) + \mm_v^2(t)),
\end{aligned}\]
where $p$ and $q$ are given by 
\[
\frac1p  = \frac12 - \frac{d/2-\eta_3}{d} \quad \mbox{and} \quad \frac1q  = \frac12 - \frac{\eta_3}{d}.
\]

\noindent When $k=2$, with $\eta_3$ above, we also choose $\eta_4\in (0,1)$ satisfying $d/2-\eta_3+\eta_4\le s$ and this gives
\[\begin{aligned}
\sfI_{33}&\le C\int_{t/2}^t e^{-\frac{t-\tau}{2}}(t-\tau)^{-1 + \frac{\eta_4}{2}}  \|\nabla^{\eta_4}(\rho(u-v))(\tau)\|_{L^2}\,d\tau\\
&\le C\int_{t/2}^t e^{-\frac{t-\tau}{2}}(t-\tau)^{-1+\frac{\eta_4}{2}}\lt( \|\nabla^{\eta_4}\rho(\tau)\|_{L^p}\|(u-v)(\tau)\|_{L^q} +\|\rho(\tau)\|_{L^p}\|\nabla^{\eta_4}(u-v)(\tau)\|_{L^q} \rt)\,d\tau\\
&\le C\int_{t/2}^t e^{-\frac{t-\tau}{2}}(t-\tau)^{-1+\frac{\eta_4}{2}} \|\rho(\tau)\|_{H^s}\lt( \|(u-v)(\tau)\|_{L^2} + \|\Delta u(\tau)\|_{L^2} + \|\Delta v(\tau)\|_{L^2}\rt)\,d\tau\\
&\le C\e_1 t^{-\frac d4 -1}(\md^0(t) + \mm_u^2(t) + \mm_v^2(t)),
\end{aligned}\]
where $p$ and $q$ are the same as above.

\noindent When $k>2$, we use $k-2 \le m <\frac d2$ to choose $\eta_5, \eta_6 \in (0,1)$ satisfying $0<d/2-(k-2)-\eta_5< d/2-(k-2)-\eta_5+\eta_6 \le s$ and get
\[\begin{aligned}
\sfI_{33}&\le C\int_{t/2}^t e^{-\frac{t-\tau}{2}}(t-\tau)^{-1 + \frac{\eta_6}{2}}  \|\nabla^{k-2+\eta_6}(\rho(u-v))(\tau)\|_{L^2}\,d\tau\\
&\le C\int_{t/2}^t e^{-\frac{t-\tau}{2}}(t-\tau)^{-1+\frac{\eta_6}{2}}\lt( \|\nabla^{k-2+\eta_6}\rho(\tau)\|_{L^p}\|(u-v)(\tau)\|_{L^q} +\|\rho(\tau)\|_{L^p}\|\nabla^{k-2+\eta_6}(u-v)(\tau)\|_{L^q} \rt)\,d\tau\\
&\le C\int_{t/2}^t e^{-\frac{t-\tau}{2}}(t-\tau)^{-1+\frac{\eta_6}{2}} \|\rho(\tau)\|_{H^s}\lt( \|\nabla^{k-2}(u-v)(\tau)\|_{L^2} + \|\nabla^k u(\tau)\|_{L^2} + \|\nabla^k v(\tau)\|_{L^2}\rt)\,d\tau\\
&\le C\e_1 t^{-\frac d4 -\frac k2}(\md^{k-2}(t) + \mm_u^k(t) + \mm_v^k(t)),
\end{aligned}\]
where $p$ and $q$ are given by 
\[
\frac1p  = \frac12 - \frac{d/2-(k-2)-\eta_5}{d} \quad \mbox{and} \quad \frac1q  = \frac12 - \frac{(k-2)+\eta_5}{d}.
\]

\noindent Hence, we have
\[
\sfI_{33} \le C\e_1t^{-\frac{d}{4}-\frac k2} \left\{\begin{array}{lll} \mm_u^2 (t) + \mm_v^2(t)+\md^0(t) &\mbox{if}& k \le 2, \\[2mm]
 \mm_u^k (t) + \mm_v^k(t)+\md^{k-2}(t) & \mbox{if} & k >2. \end{array} \right.
\]
Thus, we gather all the estimates for $\sfI_{i}$'s to attain that if $t>1$, 
\bq\label{high_t_est}
t^{\frac d4 + \frac k2} \|\nabla^k v(t)\|_{L^2} \le C\|v_0\|_{L^1} + C\e_1^2  + C\mm_v(t)^2 + C\e_1 \lt( \md(t) + \mm_u(t) + \mm_v(t)\rt).
\eq

We next consider the case (ii) $t \leq 1$.

\noindent $\diamond$ (Case (ii): estimate of $\sfI_{31}$) Since we assume $t\le 1$, one gets
\[\begin{aligned}
\sfI_{31} &\le Ct^{-\frac d4 -\frac k2} \int_0^{t/2} \|\rho(\tau)\|_{L^2}\|(u-v)(\tau)\|_{L^2}\,d\tau \le C \e_1^2 t^{-\frac d4 - \frac k2+1} \le C\e_1^2 t^{-\frac d4 - \frac k2}.
\end{aligned}\]

\noindent $\diamond$ (Case (ii): estimate of $\sfI_{32}$) Since $t \le 1$, we choose $\zeta_1>0$ satisfying $\zeta_1<\frac d2$ so that 
\[\begin{aligned}
\sfI_{32}&\le C\int_{t/2}^t \| |\xi|^{- \zeta_1}\mathds{1}_{\{ |\xi|\le 1\}} \rwhat{\rho(u-v)}(\tau)\|_{L^2}\,d\tau\\
&\le C\int_{t/2}^t  \||\xi|^{-\zeta_1}\mathds{1}_{\{|\xi|\le1\}}\|_{L^2} \|\rho(\tau)\|_{L^2}\|(u-v)(\tau)\|_{L^2}\,d\tau\\
&\le C\e_1^2 t\\
&\le C\e_1^2 t^{-\frac d4 -\frac k2}.
\end{aligned}\]

\noindent $\diamond$ (Case (ii): estimate of $\sfI_{33}$)
For $\sfI_{33}$, we again use $(1+|\xi|^2)/2 \le |\xi|^2$ when $|\xi|>1$ and again estimate $\sfI_{33}$ depending on $k$. Suppose $k < 1$. Then we choose $\zeta_2 \in (0,1)$ satisfying $d/2-\zeta_2 < s$ and obtain
\[
\begin{aligned}
\sfI_{33}&\le C \int_{t/2}^t \lt\||\xi|^k \mathds{1}_{\{|\xi|>1 \}}e^{-\frac{(1+|\xi|^2)}{2}(t-\tau)} \rwhat{\rho (u-v)}(\tau) \rt\|_{L^2}\,d\tau\\
&\le C\int_{t/2}^t e^{-\frac{t-\tau}{2}}(t-\tau)^{-\frac k2}  \|(\rho(u-v))(\tau)\|_{L^2}\,d\tau\\
&\le C\int_{t/2}^t e^{-\frac{t-\tau}{2}}(t-\tau)^{-\frac k2} \|\rho(\tau)\|_{L^p}\|(u-v)(\tau)\|_{L^q}\,d\tau\\
&\le C\int_{t/2}^t e^{-\frac{t-\tau}{2}}(t-\tau)^{-\frac k2} \|\rho(\tau)\|_{H^s}\lt( \|(u-v)(\tau)\|_{L^2} + \|\nabla u(\tau)\|_{L^2} + \|\nabla v(\tau)\|_{L^2}\rt)\,d\tau\\
&\le C\e_1^2 + C\e_1\lt(\int_{t/2}^t e^{-(t-\tau)}(t-\tau)^{-k}\,d\tau \rt)^{1/2} \lt(\int_{t/2}^t \|\nabla v(\tau)\|_{L^2}^2\,d\tau \rt)^{1/2}\\
&\le C\e_1^2 \\
&\le C\e_1^2 t^{-\frac d4 - \frac k2},
\end{aligned}
\]
where $p$ and $q$ are given by 
\[
\frac1p  = \frac12 - \frac{d/2-\zeta_2}{d} \quad \mbox{and} \quad \frac1q  = \frac12 - \frac{\zeta_2}{d}.
\]

\noindent When $1 \le k <2$, we also choose the same $\zeta_2>0$ so that
\[
\begin{aligned}
\sfI_{33}&\le C\int_{t/2}^t e^{-\frac{t-\tau}{2}}(t-\tau)^{-\frac k2}  \|(\rho(u-v))(\tau)\|_{L^2}\,d\tau\\
&\le C\int_{t/2}^t e^{-\frac{t-\tau}{2}}(t-\tau)^{-\frac k2} \|\rho(\tau)\|_{L^p}\|(u-v)(\tau)\|_{L^q}\,d\tau\\
&\le C\int_{t/2}^t e^{-\frac{t-\tau}{2}}(t-\tau)^{-\frac k2} \|\rho(\tau)\|_{H^s}\lt( \|(u-v)(\tau)\|_{L^2} + \|\nabla^k u(\tau)\|_{L^2} + \|\nabla^k v(\tau)\|_{L^2}\rt)\,d\tau\\
&\le C\e_1^2 + C\e_1t^{-\frac d4 - \frac k2} \lt( \mm_u^k(t) + \mm_v^k(t)\rt)\\
&\le C\e_1^2 t^{-\frac d4 - \frac k2} + C\e_1t^{-\frac d4 - \frac k2} \lt( \mm_u^k(t) + \mm_v^k(t)\rt),
\end{aligned}
\]
where $p$ and $q$ are the same as before. 

\noindent When $k=2$, with $\zeta_2$ above, we also choose $\zeta_3\in (0,1)$ satisfying $d/2-\zeta_2+\zeta_3\le s$ and this gives
\[
\begin{aligned}
\sfI_{33}&\le C\int_{t/2}^t e^{-\frac{t-\tau}{2}}(t-\tau)^{-1 + \frac{\zeta_3}{2}}  \|\nabla^{\zeta_3}(\rho(u-v))(\tau)\|_{L^2}\,d\tau\\
&\le C\int_{t/2}^t e^{-\frac{t-\tau}{2}}(t-\tau)^{-1+\frac{\zeta_3}{2}}\lt( \|\nabla^{\zeta_3}\rho(\tau)\|_{L^p}\|(u-v)(\tau)\|_{L^q} +\|\rho(\tau)\|_{L^p}\|\nabla^{\zeta_3}(u-v)(\tau)\|_{L^q} \rt)\,d\tau\\
&\le C\int_{t/2}^t e^{-\frac{t-\tau}{2}}(t-\tau)^{-1+\frac{\zeta_3}{2}} \|\rho(\tau)\|_{H^s}\lt( \|(u-v)(\tau)\|_{L^2} + \|\Delta u(\tau)\|_{L^2} + \|\Delta v(\tau)\|_{L^2}\rt)\,d\tau\\
&\le C\e_1 t^{-\frac d4 -1}(\md^0(t) + \mm_u^2(t) + \mm_v^2(t)),
\end{aligned}
\]
where $p$ and $q$ are still the same as before.

\noindent When $k>2$, since the estimates for $t>1$ did not use nor require $t>1$, we have
\[
\sfI_{33}\le C\e_1 t^{-\frac d4 -\frac k2}(\md^{k-2}(t) + \mm_u^k(t) + \mm_v^k(t)).
\]

\noindent Thus, we have
\[
\sfI_{33} \le C\e_1t^{-\frac{d}{4}-\frac k2} \left\{\begin{array}{lll}\e_1 + \mm_u^k(t) + \mm_v^k(t) &\mbox{if}& k < 2\\[2mm]
 \mm_u^k (t) + \mm_v^k(t)+\md^{k-2}(t) & \mbox{if} & k \ge 2. \end{array} \right.
\]
Therefore, we collect all the estimates for $\sfI_{i}$'s to attain that we still have the estimates \eqref{high_t_est} when $t\le 1$, and this gives the desired result.
\end{proof}

\begin{lemma}\label{diff_temp2}
For $t>0$ and $\theta \in [0,  \,\min\{m-2, \frac d2\}]$, we have
\[\begin{aligned}
t^{\frac{d}{4}+\frac\theta2+1}\| \nabla^\theta (u-v)(t)\|_{L^2} &\le C\|\nabla^\theta ( u_0-v_0)\|_{L^2}+ C \lt( [\mm_u(t)]^2 +  [\mm_v(t)]^2  \rt)\\
&\quad + C\e_1 + C\mm_v(t)+ C\e_1 \lt(\tilde{\md}(t){\bf 1}_{ \{ m \geq \frac d2 + 2 \}} +\md(t)\rt).
\end{aligned}\]
In particular, this shows
\[\begin{aligned}
\md(t) \le C\md(0)+ C \lt( [\mm_u(t)]^2 +  [\mm_v(t)]^2  \rt) + C\e_1 + C\mm_v(t)+ C\e_1\tilde{\md}(t){\bf 1}_{ \{ m \geq \frac d2 + 2 \}}.
\end{aligned}\]
\end{lemma} 
\begin{proof} For the same reason as in Lemma \ref{v_vel2}, here we provide the details of the proof only for the case $m < \frac d2 + 2$. The other case, $m \geq \frac d2 + 2$, will be discussed in Appendix \ref{app_lem2}.
We first observe
\[\begin{aligned}
\rwhat{(u-v)} (t) &= \rwhat{(u_0-v_0)}e^{-t} -\int_0^t e^{-(t-\tau)}\rwhat{(u\cdot \nabla) u}(\tau)\,d\tau - \int_0^t e^{-(t-\tau)} P(\xi)\rwhat{(v \cdot \nabla) v}(\tau)\,d\tau\\
&\quad  - \int_0^t e^{-(t-\tau)}|\xi|^2 \hat v(\tau)\,d\tau - \int_0^t e^{-(t-\tau)}P(\xi) \rwhat{\rho(u-v)}(\tau)\,d\tau. 
\end{aligned}\]
This yields
\begin{align*}
\||\xi|^\theta \rwhat{(u-v)}(t)\|_{L^2} &\le Ce^{-t} \||\xi|^\theta \rwhat{(u_0-v_0)}\|_{L^2} + C\int_0^t e^{-(t-\tau)}\| |\xi|^\theta \rwhat{ (u\cdot \nabla) u}(\tau)\|_{L^2}\,d\tau\\
&\quad + C\int_0^t e^{-(t-\tau)}\||\xi|^\theta \rwhat{(v\cdot \nabla) v}(\tau)\|_{L^2}\,d\tau + \int_0^t e^{-(t-\tau)}\||\xi|^{\theta+2} \hat v(\tau)\|_{L^2}\,d\tau\\
&\quad + C\int_0^t e^{-(t-\tau)} \||\xi|^\theta \rwhat{\rho(u-v)}(\tau)\|_{L^2}\,d\tau\\
&=:  Ce^{-t} \||\xi|^\theta \rwhat{(u_0-v_0)}\|_{L^2} + \sum_{i=1}^4 \sfJ_i.
\end{align*}

Similarly as in the proof of Lemma \ref{v_vel2}, we separately estimate the terms $J_i,i=1,\dots,4$ by dealing with two cases: Case (i) $t>1$ and Case (ii) $t \le 1$.

\noindent $\bullet$ (Case (i): $t>1$)
For $\sfJ_1$, note that we have
\[
\tau^{\frac d2}\|\hat u\|_{L^1} \le C \tau^{\frac d2} \||\xi|^{s+1} \hat u\|_{L^2}^{\frac{d/2}{s+1}}\|\hat u\|_{L^2}^{1-\frac{d/2}{s+1}} =C \lt( \tau^{\frac d4 + \frac{s+1}{2}}\||\xi|^{s+1}\hat u\|_{L^2}\rt)^{\frac{d/2}{s+1}} \lt(\tau^{\frac d4}\| \hat u\|_{L^2} \rt)^{1-\frac{d/2}{s+1}}
\]
and
\[
\tau^{\frac d2+\frac12}\| |\xi| \hat u\|_{L^1} \le C \tau^{\frac d2+\frac12} \||\xi|^m \hat u\|_{L^2}^{\frac{d/2+1}{m}}\|\hat u\|_{L^2}^{1-\frac{d/2+1}{m}} =C \lt( \tau^{\frac d4 + \frac{m}{2}}\||\xi|^m\hat u\|_{L^2}\rt)^{\frac{d/2+1}{m}} \lt(\tau^{\frac d4}\| \hat u\|_{L^2} \rt)^{1-\frac{d/2+1}{m}}.
\]
With this, we use Lemma \ref{lem_moser} (vi) to get
\[
\begin{aligned}
\sfJ_1 &\le C\int_0^t e^{-(t-\tau)}\lt( \| \nabla^\theta u(\tau)\|_{L^2} \|\nabla u(\tau)\|_{L^\infty} + \|u(\tau)\|_{L^\infty}\|\nabla^{\theta+1}u(\tau)\|_{L^2}\rt) d\tau\\
&\le C\e_1^2 e^{-\frac t2} + Ct^{-\frac d4 - \frac\theta2 -1} \lt(\e_1^{1-\frac{2}{d+1}}[\mm_u(t)]^{\frac{2}{d+1}}\mm_u^\theta(t) +  \e_1^{1-\frac1d}\lt[\mm_u (t)\rt]^{\frac1d} \mm_u^{\theta+1}(t)\rt) \int_{t/2}^t e^{-(t-\tau)}\,d\tau\\
&\le C\e_1^2t^{-\frac d4 -\frac \theta2 -1} + Ct^{-\frac d4 - \frac \theta2 -1}[\mm_u(t)]^2.
\end{aligned}
\]
For $\sfJ_2$, we use the estimates for $\sfI_2$ in the proof of Lemma \ref{v_vel2} for the integral on $[t/2, t]$ and  \eqref{v_inf_temp} to have
\[\begin{aligned}
\sfJ_2 &\le C\int_0^{t/2}e^{-(t-\tau)}\|\nabla^{\theta+1}v(\tau)\|_{L^2}\|v(\tau)\|_{L^\infty}\,d\tau +C\e_1^{1-\frac1d}t^{-\frac d4 -\frac \theta2 -1} \lt[\mm_v(t)\rt]^{1+\frac1d} \\
&\le C\int_0^{t/2}e^{-(t-\tau)}\|\nabla v(\tau)\|_{H^s}^{1+\frac{d/2}{s+1}}\|v(\tau)\|_{L^2}^{1-\frac{d/2}{s+1}}\,d\tau + C\e_1^{1-\frac1d}t^{-\frac d4 -\frac \theta2 -1} \lt[\mm_v(t)\rt]^{1+\frac1d} \\
&\le C\e_1^2 e^{-\frac t2}+ C\e_1^{1-\frac1d}t^{-\frac d4 -\frac \theta2 -1} \lt[\mm_v(t)\rt]^{1+\frac1d}.
\end{aligned}\]
For $\sfJ_3$,  we use $2\le \theta+2 \le m$ to get
\[\begin{aligned}
\sfJ_3 &\le Ce^{-\frac t2}\int_0^{t/2} \|\nabla^2 v(\tau)\|_{H^{m-2}}\,d\tau + C\mm_v^{\theta+2}(t)\int_{t/2}^t e^{-(t-\tau)}\tau^{-\frac{d}{4}-\frac \theta2-1}\,d\tau\\
&\le C\e_1 e^{-\frac t2} + Ct^{-\frac{d}{4}-\frac \theta2-1}\mm_v (t).
\end{aligned}\]
For $\sfJ_4$,  if $\theta<1$, we choose $\nu_1\in(\theta,2)$ and $\nu_2 \in (0,1)$ satisfying $0<d/2-\nu_1+\theta \le s$ and $d/2-\nu_2 \le s$, respectively, so that 
\[
\begin{aligned}
\sfJ_4&\le C\int_0^t e^{-(t-\tau)}\lt(\|\nabla^\theta \rho(\tau)\|_{L^{p_1}}\|(u-v)(\tau)\|_{L^{q_1}}  + \|\rho(\tau)\|_{L^{p_2}}\| \nabla^\theta(u-v)(\tau)\|_{L^{q_2}}\rt) d\tau\\
&\le C \int_0^t e^{-(t-\tau)}\|\rho(\tau)\|_{H^s}\lt(\| \nabla^\theta(u-v)(\tau)\|_{L^2} + \|\nabla^{\theta+2} u(\tau)\|_{L^2} + \|\nabla^{\theta+2}v(\tau)\|_{L^2}\rt) d\tau\\
&\le C\e_1\lt( \int_0^{t/2} + \int_{t/2}^t\rt)e^{-(t-\tau)}\lt(\| \nabla^\theta(u-v)(\tau)\|_{L^2} + \|\nabla^{\theta+2} u(\tau)\|_{L^2} + \|\nabla^{\theta+2}v(\tau)\|_{L^2}\rt) d\tau\\
&\le C\e_1^2 e^{-\frac t2} + C\e_1 t^{-\frac d4 - \frac\theta2 -1}\lt(\md^\theta(t) + \mm_u^{\theta+2}(t) + \mm_v^{\theta+2}(t) \rt),
\end{aligned}
\]
where 
\[
\frac{1}{p_i} = \frac12 - \frac{d/2-\nu_i}{d} \quad \mbox{and} \quad \frac{1}{q_i} = \frac12 - \frac{\nu_i}{d} \quad \mbox{for} \quad i=1,2.
\]

\noindent When $\theta \ge 1$ (which excludes $d=2$), since $\theta \le m-2 <s \le m-1 <\frac d2+1$, we have $s-\theta < \frac d2$ and thus,
\[
\begin{aligned}
\sfJ_4&\le C\int_0^t e^{-(t-\tau)}\lt(\|\nabla^\theta \rho(\tau)\|_{L^{p_1}}\|(u-v)(\tau)\|_{L^{q_1}}  + \|\rho(\tau)\|_{L^{p_2}}\| \nabla^\theta(u-v)(\tau)\|_{L^{q_2}}\rt) d\tau\\
&\le C \int_0^t e^{-(t-\tau)}\|\rho(\tau)\|_{H^s}\lt(\| \nabla^{d/2-s+\theta}(u-v)(\tau)\|_{L^2} + \| \nabla^{\theta + \nu_2}(u-v)(\tau)\|_{L^2}\rt) d\tau\\
&\le C\e_1\lt( \int_0^{t/2} + \int_{t/2}^t\rt)e^{-(t-\tau)}\lt(\| \nabla^\theta(u-v)(\tau)\|_{L^2} + \|\nabla^{\theta+2} u(\tau)\|_{L^2} + \|\nabla^{\theta+2}v(\tau)\|_{L^2}\rt) d\tau\\
&\le C\e_1^2 e^{-\frac t2}+C\e_1 t^{-\frac d4 - \frac\theta2 -1}\lt(\md^\theta(t) + \mm_u^{\theta+2}(t) + \mm_v^{\theta+2}(t) \rt),
\end{aligned}
\]
where $p_2$, $q_2$ and $\nu_2$ are the same as before, and 
\[
\frac{1}{p_1} = \frac12 - \frac{s-\theta}{d} \quad \mbox{and} \quad \frac{1}{q_1} = \frac12 - \frac{d/2-s+\theta}{d},
\]
and note that $d/2-s+\theta <s+1$.

Hence, we combine all the estimates for $\sfJ_i$'s to yield, when $t>1$,
\bq\label{high_t_est2}
t^{\frac d4 + \frac \theta2 +1} \|\nabla^\theta(u-v)(t)\|_{L^2} \le C\|u_0-v_0\|_{H^{m-2}} + C\e_1 + C(\mm_u^2 (t) + \mm_v^2(t)) + C\e_1 \md(t) + C\mm_v(t).
\eq

\noindent $\bullet$ (Case (ii): $t \le 1$) For $\sfJ_1$, we can simply use the well-posedness result and the estimates for the case $t>1$ to get
\[
\begin{aligned}
\sfJ_1 &\le C\int_0^t e^{-(t-\tau)}\lt( \| \nabla^\theta u(\tau)\|_{L^{\frac{d}{\theta}}} \|\nabla u(\tau)\|_{L^{\frac{1}{\frac12 -\frac{\theta}{d}}}} + \|u(\tau)\|_{L^\infty}\|\nabla^{\theta+1}u(\tau)\|_{L^2}\rt) d\tau\\
&\le C\int_0^t e^{-(t-\tau)}\lt( \| \nabla^{\frac d2} u(\tau)\|_{L^2} +\|u(\tau)\|_{L^\infty}\rt) \|\nabla^{\theta+1} u(\tau)\|_{L^2} \,d\tau\\
&\le C\e_1^2\\
&\le C\e_1^2 t^{-\frac d4 -\frac \theta2-1}.
\end{aligned}
\]
For $\sfJ_2$, we again use \eqref{v_inf_temp} to obtain
\[
\begin{aligned}
\sfJ_2 &\le C\int_0^t e^{-(t-\tau)}) \|\nabla^{\theta+1} v(\tau)\|_{L^2}\|v(\tau)\|_{L^\infty} \,d\tau\\
&\le C\int_0^t e^{-(t-\tau)} \|\nabla v(\tau)\|_{H^s}^{1+\frac{d/2}{s+1}}\|v(\tau)\|_{L^2}^{1-\frac{d/2}{s+1}} \,d\tau\\
&\le C\e_1^2\\
&\le C\e_1^2 t^{-\frac d4 -\frac \theta2-1}.
\end{aligned}
\]
\noindent For $\sfJ_3$,  we also have
\[
\sfJ_3 \le C\e_1 t^{-\frac d4 -\frac \theta2-1} + Ct^{-\frac{d}{4}-\frac \theta2-1}\mm_v (t).
\]
For $\sfJ_4$,   we follow the same estimates for the case $t>1$ to obtain
\[\begin{aligned}
\sfJ_4 &\le C\e_1^2 e^{-\frac t2} + C\e_1 t^{-\frac d4 - \frac\theta2 -1}\lt(\md^\theta(t) + \mm_u^{\theta+2}(t) + \mm_v^{\theta+2}(t) \rt)\\
&\le C\e_1^2 t^{-\frac d4-\frac\theta2 -1} + C\e_1 t^{-\frac d4 - \frac\theta2 -1}\lt(\md^\theta(t) + \mm_u^{\theta+2}(t) + \mm_v^{\theta+2}(t) \rt).
\end{aligned}
\]
Thus we gather all the estimates for $\sfJ_i$'s to conclude that when $t>1$,
\[
t^{\frac d4 + \frac \theta2 +1} \|\nabla^\theta(u-v)(t)\|_{L^2} \le C\|u_0-v_0\|_{H^{m-2}} + C\e_1 + C\e_1 \md(t) + C\mm_v(t).
\]
Therefore, we combine the above estimates with \eqref{high_t_est2} to get the desired relation.
\end{proof}

In the following three lemmas, we assume $ m \geq \frac d2 + 2$. We begin with the bound estimate of $\Delta v(t)$ in $L^\infty$.

\begin{lemma}\label{v_vel_inf}
For $t>0$, we have
\[\begin{aligned}
t^{\frac d2+1} \||\xi|^2 \hat v(t)\|_{L^1} &\le C\|v_0\|_{L^1} + C\e_1^{1-\frac1d} \lt(\e_1^{1+\frac1d}+ [\mm_v(t)]^{1+\frac1d} + \tilde{\mm}_v(t) [\mm_v(t)]^{\frac 1d} \rt)\\
&\quad + C\e_1(\e_1 + \mm_u(t) + \mm_v(t) + \tilde{\md}(t)).
\end{aligned}\]
\end{lemma}
\begin{proof}
From \eqref{v_trans}, one has
\[\begin{aligned}
\| |\xi|^2 \hat v(t)\|_{L^1} &\le \| |\xi|^2 e^{-|\xi|^2 t }\hat v_0\|_{L^1} + \int_0^t \| |\xi|^2 e^{-|\xi|^2 (t-\tau)} P(\xi) \rwhat{v \cdot \nabla v}(\tau)\|_{L^1}\,d\tau\\
&\quad +  \int_0^t \| |\xi|^2 e^{-|\xi|^2 (t-\tau)} P(\xi) \rwhat{\rho(u-v)}(\tau)\|_{L^1}\,d\tau\\
& =: \sfK_1 + \sfK_2 + \sfK_3.
\end{aligned}\]
First, we directly have
\[
\sfK_1 \le Ct^{-\frac{d}{2}-1} \|v_0\|_{L^1}.
\]
For $\sfK_2$, we use \eqref{v_inf_temp} to get
\[\begin{aligned}
\sfK_2 &\le C\int_0^{t/2} (t-\tau)^{-\frac d2-\frac 32} \| (v \otimes  v)(\tau)\|_{L^1}\,d\tau + C\int_{t/2}^t (t-\tau)^{-\frac12} \| |\xi|^2 \rwhat{v \otimes v}(\tau)\|_{L^1}\,d\tau\\
&\le C\e_1^{2-\frac 2d} t^{-\frac d2 -1}[\mm_v^0(t)]^{\frac 2d} + C\e_1^{1-\frac 1d} t^{-\frac d2 -1} \tilde{\mm}_v(t) \lt[ \mm_v(t)\rt]^{\frac 1d} \int_{t/2}^t (t-\tau)^{-\frac12} \tau^{-\frac12}\,d\tau \\
&\le C\e_1^{2-\frac 2d} t^{-\frac d2 -1}[\mm_v^0(t)]^{\frac 2d}  + C\e_1^{1-\frac 1d} t^{-\frac d2 - 1} \tilde{\mm}_v(t) [\mm_v(t)]^{\frac 1d}.
\end{aligned}\] 
For $\sfK_3$, we split the estimates as
\[\begin{aligned}
\sfK_3 &\le C\int_0^{t/2} (t-\tau)^{-\frac d2-1} \|(\rho(u-v))(\tau)\|_{L^1}\,d\tau\\
&\quad + C\int_{t/2}^t \lt\| \mathds{1}_{\{ |\xi|\le 1\}}|\xi|^2 e^{-|\xi|^2 (t-\tau)} P(\xi) \rwhat{\rho(u-v)}(\tau) \rt\|_{L^1} d\tau\\
&\quad + C\int_{t/2}^t \lt\| \mathds{1}_{\{ |\xi|>1\}}|\xi|^2 e^{-|\xi|^2 (t-\tau)} P(\xi) \rwhat{\rho(u-v)}(\tau) \rt\|_{L^1}  d\tau\\
&=: \sfK_{31} + \sfK_{32} + \sfK_{33}.
\end{aligned}\]
For $\sfK_{31}$, since we have $u-v\in L^1(\R; L^2(\R^d))$, we immediately get
\[
\sfK_{31}\le C\e_1^2 t^{-\frac d2 -1}. 
\]
For $\sfK_{32}$, when $t>1$, we choose $\zeta_1>0$ satisfying $2\zeta_1<d$ so that
\[\begin{aligned}
\sfK_{32}&\le C\int_{t/2}^{t-1/4} (t-\tau)^{-1 - \frac{\zeta_1}{2}} \| |\xi|^{-\zeta_1}\mathds{1}_{\{ |\xi|\le 1\}} \rwhat{\rho(u-v)}(\tau)\|_{L^1}\,d\tau\\
&\quad + C\int_{t-1/4}^t \||\xi|^{-\zeta_1} \mathds{1}_{\{ |\xi|\le 1\}}\rwhat{\rho(u-v)}(\tau)\|_{L^1}\,d\tau \\
&\le C\tilde{\md} (t) \int_{t/2}^{t-1/4}  (t-\tau)^{- 1 - \frac{\zeta_1}{2}}  \tau^{-\frac d2-1} \| |\xi|^{-\zeta_1} \mathds{1}_{\{ |\xi|\le 1\}}\|_{L^2}\|\rho(\tau)\|_{L^2}\,d\tau\\
&\quad + C\tilde{\md} (t) \int_{t-1/4}^t  \tau^{-\frac d2-1} \| |\xi|^{-\zeta_1} \mathds{1}_{\{ |\xi|\le 1\}}\|_{L^2}\|\rho(\tau)\|_{L^2}\,d\tau\\
&\le C\e_1 t^{-\frac d2-1}\tilde{\md}(t).
\end{aligned}\]
When $t \le 1$, one gets
\[\begin{aligned}
\sfK_{32} &\le C\int_{t/2}^t \|\hat {\rho}(\tau)\|_{L^1} \|\rwhat{(u-v)}(\tau)\|_{L^1}\,d\tau \le C\e_1^2 t \le C\e_1^2 t^{-\frac d2 -1}.
\end{aligned}\]
In either case,
\[
\sfK_{32} \le C\e_1^2 t^{-\frac d2 -1} + C\e_1 t^{-\frac d2 -1}\tilde{\md}(t).
\]
We next use $(1+|\xi|^2)/2 \le |\xi|^2$ when $|\xi|>1$ and again choose a sufficiently small $\zeta_2 \in (0,1)$ satisfying $s>\frac d2+\zeta_2$ to estimate $\sfK_{33}$ as
\[
\begin{aligned}
\sfK_{33}&\le C\int_{t/2}^t e^{-\frac{t-\tau}{2}}(t-\tau)^{-1 +\frac{\zeta_2}{2}} \| |\xi|^{\zeta_2}\mathds{1}_{ \{ |\xi|>1\}} \rwhat{\rho(u-v)}(\tau)\|_{L^1}\,d\tau\\
&\le C\int_{t/2}^t e^{-\frac{t-\tau}{2}}(t-\tau)^{-1 +\frac{\zeta_2}{2}} \lt( \| |\xi|^{\zeta_2} \hat \rho(\tau)\|_{L^1} \|  \rwhat{(u-v)}(\tau)\|_{L^1} + \|\hat\rho(\tau)\|_{L^1}\||\xi|^{\zeta_2}\rwhat{(u-v)}(\tau)\|_{L^1}\rt) d\tau\\
&\le C\int_{t/2}^t e^{-\frac{t-\tau}{2}}(t-\tau)^{-1 +\frac{\eta_2}{2}} \| \rho(\tau)\|_{H^s} \lt( \| \rwhat{(u-v)}(\tau)\|_{L^1}\rt. \cr
&\hspace{7cm} \lt. + \||\xi|^{\frac d2+2\zeta_2} \rwhat{(u-v)}(\tau) \|_{L^2}^{1/2}\| |\xi|^{\frac d2}\rwhat{(u-v)}(\tau)\|_{L^2}^{1/2}\rt) d\tau\\
&\le C\e_1 \int_{t/2}^t e^{-\frac{t-\tau}{2}}(t-\tau)^{-1 +\frac{\eta_2}{2}}  \lt( \| \rwhat{(u-v)}(\tau)\|_{L^1} + \|\nabla^{\frac d2}(u-v)(\tau)\|_{L^2} \rt.\cr
&\hspace{8cm} \lt. + \|\nabla^{\frac d2+2} u(\tau)\|_{L^2} + \|\nabla^{\frac d2+2} v(\tau)\|_{L^2}\rt) d\tau\\
&\le C\e_1 t^{-\frac d2-1} ( \tilde{\md}(t) + \md(t) + \mm_u(t) + \mm_v(t)).
\end{aligned}
\]
Thus, we gather all the estimates for $\sfK_{i}$'s to yield the desired result.
\end{proof}

We next deal with the $L^\infty$-decay of the difference $(u-v)(t)$.

\begin{lemma}\label{diff_temp_inf}
For $t>0$, we have
\[\begin{aligned}
t^{\frac d2+1}\| \rwhat{(u-v)}(t)\|_{L^1} &\le C\|\rwhat{(u_0-v_0)}\|_{L^1} + C\e_1+ C\tilde{\mm}_v(t) + C\lt( [\mm_u(t)]^2 + [\mm_v(t)]^2\rt) + C\e_1 \tilde{\md}(t).
\end{aligned}\]
\end{lemma} 
\begin{proof}
We first have
\begin{align*}
\|\rwhat{(u-v)}(t)\|_{L^1} &\le e^{-t} \|\rwhat{(u_0-v_0)}\|_{L^1} + C\int_0^t e^{-(t-\tau)}\|\rwhat{ (u\cdot \nabla) u}(\tau)\|_{L^1}\,d\tau\\
&\quad + C\int_0^t e^{-(t-\tau)}\|\rwhat{(v\cdot \nabla) v}(\tau)\|_{L^1}\,d\tau + \int_0^t e^{-(t-\tau)}\||\xi|^2 \hat v(\tau)\|_{L^1}\,d\tau\\
&\quad + C\int_0^t e^{-(t-\tau)} \| \rwhat{\rho(u-v)}(\tau)\|_{L^1}\,d\tau\\
&=:  e^{-t}\|\rwhat{(u_0-v_0)}\|_{L^1} + \sum_{i=1}^4 \sfL_i. 
\end{align*}
Here we estimate $\sfL_1$ as
\[
\begin{aligned}
\sfL_1 &\le Ce^{-\frac t2}\int_0^{t/2}\|\hat u(\tau)\|_{L^1}\|\rwhat{\nabla u}(\tau)\|_{L^1}\,d\tau + C\e_1^{1-\frac1d}[\mm_u(t)]^{1+\frac 1d}  \int_{t/2}^t e^{-(t-\tau)}\tau^{-\frac d2-1}\,d\tau\\
&\le C\e_1^2 e^{-\frac t2} + C\e_1^{1-\frac1d}t^{-\frac d2-1}[\mm_u(t)]^{1+\frac 1d}.
\end{aligned}
\]
For $\sfL_2$, we use $s>\frac d2$ and Theorem \ref{main_thm1} to get
\[\begin{aligned}
\sfL_2 &\le Ce^{-\frac t2} \int_0^{t/2} \|\hat v(\tau)\|_{L^1} \|\rwhat{\nabla  v}(\tau)\|_{L^1}\,d\tau + C\e_1^{1-\frac1d}t^{-\frac d2-1}[\mm_v(t)]^{1+\frac 1d}\\
&\le C\e_1e^{-\frac t2} \int_0^{t/2}\|\rwhat{\nabla  v}(\tau)\|_{L^1}\,d\tau  + C\e_1^{1-\frac1d}t^{-\frac d2-1}[\mm_v(t)]^{1+\frac 1d}\\
&\le C\e_1^2 e^{-\frac t2} + + C\e_1^{1-\frac1d}t^{-\frac d2-1}[\mm_v(t)]^{1+\frac 1d}.
\end{aligned}\]
For $\sfL_3$, we easily get
\[\begin{aligned}
\sfL_3 &\le Ce^{-\frac t2}\int_0^{t/2} \|\nabla^2 v(\tau)\|_{H^{m-2}}\,d\tau + C\tilde{\mm}_v(t)\int_{t/2}^t e^{-(t-\tau)}\tau^{-\frac d2-1}\,d\tau\\
&\le C\e_1 e^{-\frac t2} + Ct^{-\frac d2-1}\tilde{\mm}_v(t).
\end{aligned}\]
For $\sfL_4$, one has
\[
\begin{aligned}
\sfL_4&\le C e^{-\frac t2}\int_0^{t/2} \|\hat\rho(\tau)\|_{L^1}\|\rwhat{(u-v)}(\tau)\|_{L^1}\,d\tau + C \tilde{\md}(t) \int_{t/2}^t e^{-(t-\tau)}\|\hat\rho(\tau)\|_{L^1}\tau^{-\frac d2-1}\,d\tau\\
&\le  C\e_1^2  t e^{-\frac t2}  + C\e_1t^{-\frac d2-1}\tilde{\md}(t).
\end{aligned}
\]
Hence, we combine all the estimates for $\sfL_i$'s to give the desired result.
\end{proof}
Finally, we show the bound on $\Delta u(t)$ in $L^\infty$.
\begin{lemma}\label{u_vel_inf}
For $t>0$, we have
\[
\sup_{\tau \in [0,t]}\tau^{\frac d2 +1}\|\Delta u(\tau)\|_{L^\infty} \le \|\Delta u_0\|_{L^\infty} + C\tilde{\mm}_v(t).
\]
\end{lemma}
\begin{proof}
One easily obtains
\[
\frac{d}{dt}\|\Delta u\|_{L^p} + \|\Delta u\|_{L^p} \le C\|\nabla u\|_{L^\infty}\|\Delta u\|_{L^p} + C\|\Delta v\|_{L^p}  \leq C\e_1 \|\Delta u\|_{L^p} + C\|\Delta v\|_{L^p}.
\]
Then one uses Gr\"onwall's lemma and let $p\to\infty$ to obtain the desired estimate.
\end{proof}

Now, we collect all the estimates up to now to yield the following result.

\subsection{Proof of Theorem \ref{main_thm2}} 
First, we use  Theorem \ref{main_thm1} to assert that for $t \le 1$,
\[
\mm_v^0(t) + \mm_u(t) + \md(t) + \tilde{\md}(t){\bf 1}_{ \{ m \geq \frac d2 + 2 \}} \le C\e_1.
\]
Thus, we apply this to Lemma \ref{v_vel2} and get
\[
\mm_v(t) \le C\e_1, \quad \mbox{for} \quad t \le 1,
\]
and in turn, those also imply
\[
\tilde{\mm}_v(t){\bf 1}_{ \{ m \geq \frac d2 + 2 \}} \le C\e_1, \quad \mbox{for} \quad t \le 1.
\]
Thus, the following set becomes non-empty:
\[
{\mathcal S} := \lt\{ t >0\ | \ \mm_v(t) + \mm_u(t) + \md(t) + \tilde{\mm}_v(t){\bf 1}_{ \{ m \geq \frac d2 + 2 \}} + \tilde{\md}(t){\bf 1}_{ \{ m \geq \frac d2 + 2 \}} <\e_1^{2/3}\rt\}.
\]
As in the case (i), we assume for contradiction that $\sup\tilde{\mathcal S} <\infty$. Then there exists $T>0$ satisfying
\[
 \mm_v(T) + \mm_u(T) + \md(T) + \tilde{\mm}_v(T){\bf 1}_{ \{ m \geq \frac d2 + 2 \}} + \tilde{\md}(T){\bf 1}_{ \{ m \geq \frac d2 + 2 \}} =\e_1^{2/3}.
\]
Then, Lemma \ref{v_vel2} implies
\bq\label{all_v_est2}
\begin{aligned}
\mm_v(T) &\le C\|v_0\|_{L^1} + C\e_1^2 + C\mm_v^0(T)\mm_v(T) + C\e_1\lt(\tilde{\md}(T){\bf 1}_{ \{ m \geq \frac d2 + 2 \}} + \md(T) + \mm_u(T) + \mm_v(T) \rt)\\
&\le C\e_1.
\end{aligned}
\eq
In turn, one applies \eqref{all_v_est2} to Lemmas \ref{u_vel}, \ref{diff_temp2} and \ref{v_vel_inf} to obtain
\bq\label{all_u_est2}
\begin{aligned}
\mm_u(T) &\le C\|u_0\|_{H^m} + C\mm_v(T) \le C\e_1,\cr
\md(T) &\le C\lt(\|u_0\|_{H^s} + \|v_0\|_{H^s} \rt) + C \lt( [\mm_u(t)]^2 +  [\mm_v(t)]^2  \rt)\\
&\quad + C\e_1 + C\mm_v(T)+ C\e_1 \lt(\tilde{\md}(T){\bf 1}_{ \{ m \geq \frac d2 + 2 \}} +\md(T)\rt)\\
&\le C\e_1,
\end{aligned}
\eq
and
\[
\begin{aligned}
\tilde{\mm}_v(T){\bf 1}_{ \{ m \geq \frac d2 + 2 \}} &\le C\|v_0\|_{L^1} + C\e_1^{1-\frac1d} \lt(\e_1^{1+\frac1d}+ [\mm_v(T)]^{1+\frac1d} + \tilde{\mm}_v(T){\bf 1}_{ \{ m \geq \frac d2 + 2 \}} [\mm_v(T)]^{\frac 1d} \rt)\\
&\quad + C\e_1(\e_1 + \mm_u(T) + \mm_v(T) + \tilde{\md}(T){\bf 1}_{ \{ m \geq \frac d2 + 2 \}})\\
&\le C\e_1 + C\e_1 \tilde{\mm}_v(T){\bf 1}_{ \{ m \geq \frac d2 + 2 \}}.
\end{aligned}
\]
Hence, we also have
\bq\label{all_v_est_inf}
\tilde{\mm}_v(T){\bf 1}_{ \{ m \geq \frac d2 + 2 \}} \le C\e_1.
\eq
Finally, our assumption and \eqref{all_v_est_inf} yield
\bq\label{all_diff_est_inf}
\begin{aligned}
\tilde{\md}(T){\bf 1}_{ \{ m \geq \frac d2 + 2 \}} &\le C\lt( \|u_0\|_{H^s} + \|v_0\|{H^s}\rt) + C\e_1 +C\tilde{\mm}_v(T){\bf 1}_{ \{ m \geq \frac d2 + 2 \}} + C\lt( [\mm_u(T)]^2 + [\mm_v(T)]^2\rt) \cr
&\le C\e_1.
\end{aligned}
\eq
Thus, we gather all the results \eqref{all_v_est2}--\eqref{all_diff_est_inf} to attain
\[ \mm_v(T) + \mm_u(T) + \md(T) + \tilde{\mm}_v(T){\bf 1}_{ \{ m \geq \frac d2 + 2 \}} + \tilde{\md}(T){\bf 1}_{ \{ m \geq \frac d2 + 2 \}} =\e_1^{2/3} \le C\e_1\]
which is a contradiction when $\e_1$ is small enough. Thus, we have $\sup {\mathcal S} = \infty$, and this gives the desired temporal decay in $L^2$ norm up to $(\frac d2 +2)$-order derivatives and the decay in $L^\infty$-norm. For the temporal decay of $L^2$-norms with higher-order derivatives than $\frac d2+2$, following the procedure up to now would give the desired estimates. 
 
%
%
%
%
 
\section*{Acknowledgments}
We thank Kyungkeun Kang for helpful conversations regarding the large time behavior estimates for the incompressible Navier--Stokes equations. The work of Y.-P. Choi was supported by NRF grant (No. 2022R1A2C1002820). The work of J. Kim was supported by a KIAS Individual Grant (MG086501) at Korea Institute for Advanced Study.

\appendix

\section{Proof of Lemmas \ref{u_vel} \& \ref{v_vel2} when $m\geq \frac d2 + 2$}\label{app_a}

\subsection{Proof of Lemma \ref{u_vel} when $m\geq \frac d2 + 2$}\label{app_lem1}

We recall from the proof of Lemma \ref{u_vel} that
\[\begin{aligned}
\| |\xi|^k \hat v(t)\|_{L^2} &\le \| |\xi|^k e^{-|\xi|^2 t }\hat v_0\|_{L^2} + \int_0^t \| |\xi|^k e^{-|\xi|^2 (t-\tau)} P(\xi) \rwhat{v \cdot \nabla v}(\tau)\|_{L^2}\,d\tau\\
&\quad +  \int_0^t \| |\xi|^k e^{-|\xi|^2 (t-\tau)} P(\xi) \rwhat{\rho(u-v)}(\tau)\|_{L^2}\,d\tau\\
& =: \sfI_1 + \sfI_2 + \sfI_3,
\end{aligned}\]
where, by using the same argument as in the proof of Lemma \ref{u_vel}, $\sfI_1$ and $\sfI_2$ can be estimated as 
\[
\sfI_1 \le Ct^{-\frac d4 - \frac k2}\|v_0\|_{L^1}
\]
and
\[
\sfI_2 \le C\e_1^{2-\frac 2d}t^{-\frac d4 - \frac k2}[\mm_v^0(t)]^{\frac 2d} + C\e_1^{\frac{s+1+d/2}{2(s+1)}}t^{-\frac d4 -\frac k2} \mm_v^k(t) [\mm_v^0(t)]^{\frac{s+1-d/2}{2(s+1)}}.
\]
For $\sfI_3$, we also split it as
\[\begin{aligned}
\sfI_3 &\le C\int_0^{t/2} (t-\tau)^{-\frac{d}{4}-\frac k2} \|(\rho(u-v))(\tau)\|_{L^1}\,d\tau\\
&\quad + C\int_{t/2}^t \lt\| \mathds{1}_{\{ |\xi|\le 1\}}|\xi|^k e^{-|\xi|^2 (t-\tau)} P(\xi) \rwhat{\rho(u-v)}(\tau) \rt\|_{L^2} d\tau\\
&\quad + C\int_{t/2}^t \lt\| \mathds{1}_{\{ |\xi|>1\}}|\xi|^k e^{-|\xi|^2 (t-\tau)} P(\xi) \rwhat{\rho(u-v)}(\tau) \rt\|_{L^2} d\tau\\
&=: \sfI_{31} + \sfI_{32} + \sfI_{33}.
\end{aligned}\]
Then still the same estimates imply
\[
\sfI_{31} \le C\e_1^2 t^{-\frac d4 - \frac k2} + C\e_1 t^{-\frac d4 - \frac k2} \md^0(t).
\]
For $\sfI_{32}$, the estimates for $k\le 2$ are the same as before. When $k>2$ and $t>1$, we obtain
\[\begin{aligned}
\sfI_{32}&\le C\int_{t/2}^{t-1/4} (t-\tau)^{-\frac k2}\|(\rho(u-v))(\tau)\|_{L^2}\,d\tau + C\int_{t-1/4}^t \|(\rho(u-v))(\tau)\|_{L^2}\,d\tau\\
&\le   C\int_{t/2}^{t-1/4} (t-\tau)^{-\frac k2}\| \rho(\tau)\|_{L^2}\|(u-v)(\tau)\|_{L^\infty}\,d\tau + C\int_{t-1/4}^t \|\rho(\tau)\|_{L^2}\|(u-v)(\tau)\|_{L^\infty}\,d\tau\\
&\le C\e_1 t^{-\frac d4 - \frac k2} \tilde{\md}(t),
\end{aligned}\] 
and when $k>2$ and $t\le 1$, one uses $s > m-2 \ge \frac d2$ and Theorem \ref{main_thm1} to get
\[\begin{aligned}
\sfI_{32}&\le C\int_{t/2}^t \|(\rho(u-v))(\tau)\|_{L^2}\,d\tau \le C\int_{t/2}^t \|\rho(\tau)\|_{L^2}\|(u-v)(\tau)\|_{L^\infty}\,d\tau \le C\e_1^2 t\le C\e_1^2 t^{-\frac d4 -\frac k2}.
\end{aligned}\]
In either case, we have
\[
\sfI_{32} \le C\e_1^2 t^{-\frac d4 - \frac k2} + C\e_1 t^{-\frac d4 - \frac k2}\lt(\md(t) + \tilde{\md}(t) \rt).
\]
For $\sfI_{33}$, we can also use the same estimates as in the proof of Lemma \ref{u_vel} when $k< \frac d2+2$. When $k=\frac d2+2$, we choose $\eta\in(0,1)$ satisfying $s>\frac d2 +\eta$ so that
\[\begin{aligned}
\sfI_{33} &\le C\int_{t/2}^t e^{-\frac {(t-\tau)}{2}}(t-\tau)^{-1+\frac \eta2} \|\nabla^{\frac d2+\eta}(\rho(u-v))(\tau)\|_{L^2}\,d\tau\\
&\le C\int_{t/2}^t e^{-\frac {(t-\tau)}{2}}(t-\tau)^{-1+\frac \eta2} \lt(\|\nabla^{\frac d2+\eta}\rho(\tau)\|_{L^2}\|(u-v)(\tau)\|_{L^\infty} + \|\rho(\tau)\|_{L^\infty}\|\nabla^{\frac d2+\eta}(u-v)(\tau)\|_{L^2} \rt) d\tau\\
&\le C\int_{t/2}^t e^{-\frac {(t-\tau)}{2}}(t-\tau)^{-1+\frac \eta2}\|\rho(\tau)\|_{H^s} \lt(\|(u-v)(\tau)\|_{L^\infty} + \|\nabla^{\frac d2}(u-v)(\tau)\|_{L^2} \rt.\cr
&\hspace{8cm} \lt. + \|\nabla^{\frac d2+2} u(\tau)\|_{L^2} + \|\nabla^{\frac d2+2} v(\tau)\|_{L^2} \rt) d\tau\\
&\le C\e_1 t^{-\frac d2 -1} \lt(\tilde{\md}(t) + \md^{\frac d2}(t) + \mm_u^{\frac d2+2}(t) + \mm_v^{\frac d2+2}(t) \rt),
\end{aligned}\]
and hence, in either case, we have
\[
\sfI_{33} \le C\e_1^2 t^{-\frac d4 -\frac k2} + C\e_1t^{-\frac d4 - \frac k2}\lt(\tilde{\md}(t) + \md(t) + \mm_u(t) + \mm_v(t)  \rt).
\]
Thus, we gather all the estimates to yield the desired relation.

\subsection{Proof of Lemma \ref{v_vel2} when $m\geq \frac d2 + 2$}\label{app_lem2} We recall from the proof of Lemma \ref{diff_temp2} that
\begin{align*}
\||\xi|^\theta \rwhat{(u-v)}(t)\|_{L^2} &\le Ce^{-t} \||\xi|^\theta \rwhat{(u_0-v_0)}\|_{L^2} + C\int_0^t e^{-(t-\tau)}\| |\xi|^\theta \rwhat{ (u\cdot \nabla) u}(\tau)\|_{L^2}\,d\tau\\
&\quad + C\int_0^t e^{-(t-\tau)}\||\xi|^\theta \rwhat{(v\cdot \nabla) v}(\tau)\|_{L^2}\,d\tau + \int_0^t e^{-(t-\tau)}\||\xi|^{\theta+2} \hat v(\tau)\|_{L^2}\,d\tau\\
&\quad + C\int_0^t e^{-(t-\tau)} \||\xi|^\theta \rwhat{\rho(u-v)}(\tau)\|_{L^2}\,d\tau\\
&=:  Ce^{-t} \||\xi|^\theta \rwhat{(u_0-v_0)}\|_{L^2} + \sum_{i=1}^4 \sfJ_i. 
\end{align*}
For $\sfJ_1$, $\sfJ_2$ and $\sfJ_3$, we can still obtain the same bound estimates:
\[\begin{aligned}
\sfJ_1 &\le C\e_1^2 t^{-\frac d4 -\frac \theta2 -1} + Ct^{-\frac d4 -\frac \theta2 -1}[\mm_u(t)]^2,\\
\sfJ_2 &\le C\e_1^2 t^{-\frac d4 -\frac \theta2 -1} + Ct^{-\frac d4 -\frac \theta2 -1}[\mm_v(t)]^2, \quad \mbox{and}\\
\sfJ_3 &\le C\e_1t^{-\frac d4 - \frac\theta2 -1} + Ct^{-\frac d4 -\frac\theta2 -1}\mm_v(t).
\end{aligned}\]
For $\sfJ_4$, when $t>1$, we use $s>m-2\ge \frac d2$ to get
\[
\begin{aligned}
\sfJ_4 &\le C\int_0^t e^{-(t-\tau)}\lt( \|\nabla^\theta \rho(\tau)\|_{L^2}\|(u-v)(\tau)\|_{L^\infty} + \|\rho(\tau)\|_{L^\infty}\|\nabla^\theta(u-v)(\tau)\|_{L^2}\rt) d\tau\\
&\le C\e_1 \lt(\int_0^{t/2} + \int_{t/2}^t \rt) e^{-(t-\tau)}\lt( \|(u-v)(\tau)\|_{L^\infty} + \|\nabla^\theta(u-v)(\tau)\|_{L^2}\rt) d\tau\\
&\le C\e_1^2 e^{-\frac t2} + C\e_1 t^{-\frac d4 - \frac\theta2 -1}\lt( \tilde{\md}(t) + \md^\theta(t)\rt),
\end{aligned}
\]
and when $t\le 1$,
\[
\begin{aligned}
\sfJ_4 &\le C\int_0^t e^{-(t-\tau)}\lt( \|\nabla^\theta \rho(\tau)\|_{L^2}\|(u-v)(\tau)\|_{L^\infty} + \|\rho(\tau)\|_{L^\infty}\|\nabla^\theta(u-v)(\tau)\|_{L^2}\rt) d\tau\\
&\le C\e_1^2 \le C\e_1^2 t^{-\frac d4 - \frac\theta2 -1}. 
\end{aligned}
\]
In either case, we have
\[
\sfJ_4 \le C\e_1^2 t^{-\frac d4 -\frac\theta2 -1} + C\e_1 t^{-\frac d4 -\frac \theta2-1}\lt( \tilde{\md}(t) + \md(t)\rt).
\]
Thus, we combine all the estimates for $\sfJ_i$'s to conclude the desired result.

%
%
%
%
%
%
%
%

\end{document}